\documentclass[10pt, english, reqno]{amsart}
\usepackage{array}
\usepackage{latexsym, amssymb, enumerate, amsmath, array}

\usepackage{mathrsfs}
\usepackage{dsfont}
\usepackage{amsfonts,amsmath,amssymb,amsbsy}
\usepackage{graphicx}

\usepackage{latexsym}

\usepackage{enumerate}
\usepackage{mathrsfs}
\usepackage{stmaryrd}
\usepackage{amsopn}
\usepackage{amsmath}
\usepackage{amsfonts}
\usepackage{amsbsy}
\usepackage{amscd,indentfirst,epsfig}
\usepackage{hyperref}

\usepackage{amsthm,psfrag}
\usepackage{dsfont}
\usepackage{colordvi}
\usepackage{pstricks}

\usepackage{pifont}
\usepackage{enumerate}

\renewcommand{\baselinestretch}{1}

\def \leq {\leqslant}
\def \geq {\geqslant}
\numberwithin{equation}{section}
\def\ind#1{\lower5pt\hbox{$\scriptstyle #1$}}

\def \X {\mathscr{E}}
\def \x {\mathbf{x}}
\def \z {\mathbf{z}}
\def\y {\mathbf{y}}
\def\lp {L^1_+}
\def\lm {L^1_-}
\def \ff {\mathscr{F}}

\def \d {\mathrm{d}}
\def \D {\mathscr{D}}
\def \G {\mathscr{G}}

\def \Be {\mathscr{B}}
\def \ml {M_{\lambda}}

\def \K {\mathcal{K}}
\def \A {\mathcal{A}}
\def \uot {(U_0(t))_{t \geq 0}}

\def \vt {(V_H(t))_{t \geq 0}}

\def \T {\mathcal{T}}
\def \B {\mathsf{B}}

\def \cc {\mathfrak{a}}
\def \t {\tau}

\def \O {\mathbf{\Omega}}

\renewcommand{\baselinestretch}{1.2}
\newtheorem{theorem}{Theorem}[section]
\newtheorem{definition}{Definition}[section]
\newtheorem{proposition}{Proposition}[section]
\newtheorem{lemma}{Lemma}[section]
\newtheorem{corollary}{Corollary}[section]

\theoremstyle{definition}
\newtheorem{remark}{Remark}[section]

\begin{document}

\title[]
{{Transport semigroup associated to positive boundary conditions of unit norm: a Dyson-Phillips approach}}

\author{Luisa Arlotti \& Bertrand Lods}

\address{\textbf{Luisa Arlotti}, Universit\`{a} degli Studi di Udine, 33100 Udine, Italy.}
\email{arlotti@uniud.it.}

\address{\textbf{Bertrand Lods},  Universit\`{a} degli
Studi di Torino \& Collegio Carlo Alberto, Department of Economics and Statistics, Corso Unione Sovietica, 218/bis, 10134 Torino, Italy.}\email{lodsbe@gmail.com}

\maketitle

\begin{abstract}
We revisit our study of general transport operator with general force field and general invariant measure by considering, in the $L^1$ setting, the linear transport operator $\T_H$ associated to a linear and positive  boundary operator $H$ of unit norm. It is known that in this case an extension of $\T_H$ generates a substochastic (i.e. positive contraction) $C_0$-semigroup $(V_H(t))_{t\geq 0}$.  We show here that $(V_H(t))_{t\geq 0}$ is the smallest substochastic $C_0$-semigroup with the above mentioned property and provides a representation of $(V_H(t))_{t \geq 0}$ as the sum of an expansion series similar to Dyson-Phillips series. We develop an honesty theory for such boundary perturbations that allows to consider the honesty of trajectories on subintervals $J \subseteq [0,\infty)$. New necessary and sufficient conditions for a trajectory to be honest are given in terms of the aforementioned series expansion.
\medskip\\
\textit{AMS Subject Classifications} (2000): 47D06, 47N55, 35F05, 82C40\\
 \textit{Key words}: Transport equation, boundary conditions,
substochastic semigroups, honesty theory.
\end{abstract}

\section{Introduction}
\label{sec1}

We investigate here the  well-posedness (in the sense
of semigroup theory) in $L^1(\O,\d\mu)$ of the general transport equation
\begin{subequations}\label{1}
\begin{equation}\label{1a}
\partial_t f(\x,t)+\mathscr{F}(\x)\cdot \nabla_\x
f(\x,t) =0 \qquad (\x \in \O , \:t
> 0),\end{equation}
supplemented by the abstract boundary condition
\begin{equation}\label{1b}
f_{|\Gamma_-}(\y,t)=H(f_{|\Gamma_+})(\y,t), \qquad \qquad (\y \in
\Gamma_-, t >0),
\end{equation}
and the initial condition
\begin{equation}\label{1c}f(\x,0)=f_0(\x), \qquad \qquad (\x \in \O). \end{equation}\end{subequations}
Here $\O$ is a sufficiently smooth open subset of $\mathbb{R}^N$ endowed with a positive Radon measure $\mu$,
$\Gamma_{\pm}$ are suitable boundaries of the phase space and the field $\mathscr{F}$ is globally Lispchitz and \textrm{divergence free} with respect to $\mu$, in the sense that $\mu$ is a measure invariant by the (globally defined) flow associated to $\mathscr{F}$. Our main concern here is the influence of the boundary conditions \eqref{1b} and we treat here the delicate case of a boundary operator
$$H \::\:L^1_+ \to L^1_-$$
which is linear, positive,  bounded ($L^1_\pm$ being suitable trace spaces
corresponding to the boundaries $\Gamma_\pm$, see Section 2 for
details) and of unit norm
\begin{equation}\label{H1} \|H\|_{\mathscr{B}(\lp,\lm)}={\sup_{f \in \lp, \|f\|_{\lp}=1}}\|H f\|_{\lm}=1.\end{equation}

Our motivation for studying such a problem is the study of kinetic equation of Vlasov-type for which
the phase space $\O$ is a cylindrical domain $\O= {D}\times
\mathbb{R}^N \subset \mathbb{R}^{2N}$ ( ${D }$ being a sufficiently
 smooth open subset of $\mathbb{R}^N$) and the field $\mathscr{F}$ is given by \begin{equation}\label{transportneu}
 \mathscr{F}(\x)=(v,\mathbf{F}(x,v)) \quad \text{ for any } \quad \x=(x,v) \in
 \O\end{equation}
$\mathbf{F}\::\:\O \to \mathbb{R}^N$ is a time independent force field. The simplest (but already very rich) example of such a kinetic equation is the so-called \emph{free-streaming} equation for which $\mathbf{F}=0$. Boundary conditions in such kinetic equations are usually modeled by a boundary operator $H$ which relates the incoming and outgoing boundary fluxes of particles; the form of this operator depends on the gas-surface interaction (see \cite{cer} for more details on such a topic).

The mathematical study of the aforementioned problem has already a long story starting from  the seminal paper
\cite{bard} who considered the case in which $\mu$ is the Lebesgue measure and the so-called `no
re-entry' boundary conditions (i.e. $H=0$ in \eqref{1b}). More
general fields and boundary conditions (but still mostly
associated with the Lebesgue measure) have been considered in \cite{beals}. The free-streaming case (i.e. $\mathscr{F}(x,v)=(v,0)$) received much more attention, starting from \cite{voigt}, where the free streaming transport operator associated to different boundary operators $H$ is deeply investigated (see also \cite{lods} for general boundary conditions). Recently,  transport operators associated to general external fields and general measures, with general bounded boundary conditions have been thoroughly investigated by the authors in collaboration with J. Banasiak in a series of papers \cite{abl1, abl2, al3}  that contain both a generalization of the theory developed in the free streaming case and some new results. Summarizing the known results on this topic, one can say that the transport operator associated to $H$, that we shall denote $\T_H$ (see Section 2 for a precise definition)  is the generator of a strongly continuous semigroup when the boundary operator $H$ is a contraction (and also for some very peculiar multiplying boundary conditions, \cite{voigt, lods, abl2,al3,boulanouar}). \medskip

A very interesting and important case, both from the mathematical and physical point of view, arises whenever $H$ is a \emph{positive boundary operator of unit norm} (see \eqref{H1}). In such a case,  one can not state \emph{a priori} that $\T_H$ generates a $C_0$-semigroup in $L^1(\O,\d\mu)$.  Nevertheless, since for each $ r \in [0,1)$ the operator $H_r := rH$ is a strict contraction, the transport operator ${\T}_{H_r}$ associated to $H_r$ does generate  a $C_0$- semigroup $(V_r(t))_{t\geq 0}$. These semigroups are \emph{substochastic, i. e. they are positive contraction semigroups} and one can show that the strong limit $\lim_{r\nearrow 1-}V_r(t) := V_H(t)$ exists and defines a $C_0$-semigroup in $L^1(\O,\d\mu)$. Its generator  $\mathcal{A}$ is then an extension of $\T_H$ and a natural question is to recognize  if  $\mathcal{A} = {\overline{\T_H}}$ or not. For conservative conditions, i.e. if
 $$\|Hf\|_{\lm}=\|f\|_{\lp} \qquad \forall f \in \lp$$
 it is known that the semigroup $(V(t))_{t\geq 0}$ is conservative if and only if
 $\mathcal{A}=\overline{\T_H}$. On the contrary, whenever $\mathcal{A} \supsetneq {\overline{\T_H}}$ a mass loss occurs, i.e. there exists nonnegative $f$ such that $\|V_H(t)f \| < \|f\|$ for some $t > 0$.

 \medskip

As first observed in \cite{all}, such a problematic is very similar to what occurs in the so-called substochastic theory of additive perturbations of semigroups, (see the monograph \cite{ba}), where one is faced with the following problem: let $(T,D(T))$ be the generator of a substochastic $C_0$-semigroup $(G_T(t))_{t\geq 0}$ in $X = L^1(\Sigma,\d\nu)$ (where $(\Sigma,\nu)$ is a given measure space) and let $(B,\D(B))$ be a non-negative linear operator in $X$ such that $\D(T) \subseteq \D(B)$ and $\int_\Sigma (T + B)f \d\nu \leq 0$ for all $f\in \D(T)_+ = \D(T)\cap X_+$. Then for any $0 < r < 1$ operator $(T + rB,\D(T))$ generates a $C_0$- semigroup $(G_r(t))_{t\geq 0}$. These semigroups are such that the strong limit $\lim\limits_{r\nearrow 1-}G_r(t) := G_K(t)$ exists and the family $(G_K(t))_{t\geq 0}$ is a $C_0$-semigroup generated by an extension $K$ of $T + B$. In the context of additive perturbations of substochastic semigroups a complete characterization of $K$ is given; it is shown that $(G_K(t))_{t\geq 0}$ is the smallest (in the lattice sense) $C_0$-semigroup  generated by an extension of $T + B$. Moreover $G_K(t)$ can be written as the sum of a strongly convergent series of linear positive operators (Dyson-Phillips expansion series) and a satisfying honesty theory, dealing with the mass carried by individual trajectories, has been developed  \cite{ba, almm, mmv}. Such a honesty theory for additive perturbation has been based mainly on the so-called \emph{resolvent approach} (i.e. on the study of the resolvent of $(\lambda-K)^{-1}$) and such a resolvent approach has been applied to the boundary perturbation case in \cite{all,mmk}. Recently a new approach to honesty has been proposed, based now on the semigroup approach and the fine properties of the Dyson-Phillips iterated \cite{almm}. Such an approach is equivalent to the resolvent one but its main interest lies in the fact that it is robust enough to be applied to other kind of problems in which the resolvent approach would be inoperative (e.g. for non-autonomous families \cite{almm1} or integrated semigroups \cite{almm2}).\medskip

In the present paper we want to recognize that a fully similar study can be carried out for the operator $\T_H$. Notice that several results concerning the transport operator $\T_H$ and the semigroup $(V_H(t))_{t \geq 0}$ are already available in the literature. A complete characterization of $\A$ is given in \cite{abl2} where it is shown that $\A$ is an extension of $\T_H$; the study of conservative boundary conditions has been performed, in the free-streaming case, in \cite{all} and, for general force fields, in \cite{abl2}. The general case of boundary operators with unit norm has been handled with in \cite{mmk} where a detailed \textit{honesty theory} has been performed. Nevertheless the obtained results are not so satisfying as those obtained in the substochastic
theory of additive perturbations of semigroups. In particular the question of whether $(V_H(t))_{t \geq 0}$ is the \emph{smallest} substochastic $C_0$-semigroup generated by an extension of $\T_H$ remains open and the honesty theory performed in \cite{abl2, mmk}  is based on the resolvent approach only.

The objective of the present paper is to fill this blank. In particular, the main novelty of the paper lies in the  following:
 \begin{enumerate}[i)]
 \item First, we prove  that indeed the semigroup $(V_H(t))_{t \geq 0}$ is the smallest (in the lattice sense) substochastic $C_0$-semigroup generated by an extension of $\T_H$.
 \item Second, and more important, we develop a 'semigroup approach' to the honesty theory of boundary perturbations, exploiting the recent results in \cite{al3} which allow  to provide a characterization of the semigroup $(V_H(t))_{t \geq 0}$ as an expansion series, similar to the Dyson-Phillips arising in the additive perturbation case. While the resolvent approach allows to establish necessary and sufficient conditions for a trajectory to be honest (i. e. honest on $[0,\infty)$) the new semigroup approach allows to establish more general necessary and sufficient conditions for a trajectory to be honest on a subinterval $J \subseteq [0,\infty)$. We strongly believe that such a semigroup approach has its own interest and that, as it occurs for additive perturbation \cite{almm1}, it could hopefully be extended to deal with non-autonomous problems.
 \end{enumerate}

 To be more precise, the contents of the paper are  as follows. In Section 2 we introduce the necessary notation and define the transport operator $\T_H$. This section is mainly taken from the recent contributions \cite{abl2,al3}. In Section 3 we establish the most important properties of the semigroup $(V_H(t))_{t\geq 0}$ and its generator, in particular showing that $(V_H(t))_{t\geq 0}$ is the smallest substochastic $C_0$-semigroup generated by an extension of $\T_H$. In Section 4 we develop the honesty theory for boundary perturbations, introducing first useful functionals and defining then the concept of honesty of trajectories on subintervals $J \subseteq [0,\infty)$. We obtain also necessary and sufficient conditions for the honesty in the spirit of \cite[Section 6]{abl2} and \cite{mmk} not only using the usual resolvent approach but also using the series approach introduced in \cite{almm}. In Section 6 two well-known examples are revisited using our new approach, that allows us to deduce new interesting properties.

\section{Preliminaries}

In the present section, we introduce the general mathematical framework we shall consider in the sequel. The material from this section is mainly taken from \cite{abl1,abl2} and we refer to these two contributions for further properties of abstract transport operators. We begin with the rigorous definition of the transport operator $\T_H$ associated to bounded boundary operator $H.$

\subsection{Definition of the transport operator ${\T}_H$.}
\label{sec2}

In this paper we consider transport operators associated to general external fields and general measures, according to the theory developed in two recent contribution \cite{abl1},\cite{abl2}. More precisely, given a smooth open subset $\O$ of $\mathbb{R}^N$, we consider  a time independent globally Lipschitz vector
field $\mathscr{F} \::\:\mathbb{R}^N \to \mathbb{R}^N$ so that, for any $\x \in \O$, the Cauchy problem
\begin{equation}\label{chara}
\frac{\d \mathbf{{X}}}{\d t}(t)=\ff(\mathbf{{X}}(t)), \qquad \forall
t \in \mathbb{R}\:;\qquad \mathbf{{X}}(0)=\x \in \O\end{equation}
admits a unique global solution
$$(\x,t) \in \O \times \mathbb{R} \longmapsto {\Phi}(\x,t) \in \mathbb{R}^N,$$
that allows to define a flow $(T_t)_{t \in \mathbb{R}} $ given by $T_t=\Phi(\cdot, t)$. As  in \cite{abl1}, we
assume that there exists a Radon measure $\mu$ over $\mathbb{R}^N$ which is  invariant under the
flow $(T_t)_{t \in \mathbb{R}}$, i.e.
\begin{equation}\label{ass:h2} \mu(T_t A) = \mu(A) \text{ for any
measurable subset } A \subset \mathbb{R}^N \text{ and any }  t \in
\mathbb{R}.\end{equation}

Of course, solutions to \eqref{chara} do not necessarily belong to
$\O$ for all times, leading to the definition of stay times of the
characteristic curves in $\O$: for any $\x \in {\O}$, define
\begin{equation}\label{stay}
\tau_{\pm}(\x)=\inf \{s > 0\,;{\Phi}(\x,\pm s) \notin {\O}\},\end{equation} with the
convention that $\inf \varnothing=\infty.$ This allows to represent $\O$ as $\O=\O_{\pm} \cup \O_{\pm \infty}$ where
$${\O}_{\pm}=\{\x \in {\O}\,;\,\t_{\pm}(\x) < \infty\}, \qquad \text{and} \quad
{\O}_{\pm\infty}=\{\x \in {\O}\,;\,\t_{\pm}(\x) = \infty\}.$$
Moreover, we define the \textit{incoming} and \textit{outgoing} boundaries as
\begin{equation}\label{gammapm} \Gamma_{\pm}:=\left\{\y \in \partial
{\O}\,;\exists \x \in {\O},\, \t_{\pm}(\x) < \infty \text{ and }
\y={\Phi}(\x,\pm
\t_{\pm}(\x))\,\right\}.\end{equation}
The definition of the stay time $\tau_{\pm}$ extends then to $\Gamma_{\pm}$ by setting simply $\tau_{\pm}(\y)=0$ and $\tau_{\mp}(\y)=\tau_+(\x)+\tau_-(\x)$ for any $\y \in \Gamma_{\pm}$ with $\y={\Phi}(\x,\pm
\t_{\pm}(\x))$. Notice that, with the above definition, $\tau_{\mp}(\y)$ is well defined (i.e. the definition does not depend on the choice of $\x \in \O_{\pm}$) and $\tau_{\mp}(\y)$ is nothing but the length of the characteristic curves having $\y$ as its left (respectively right) end-point. We finally set
$$\Gamma_{\pm\infty}=\{\y \in \Gamma_{\pm}\,;\, \t_{\mp}(\y) = \infty\}.$$
With such notations, one can prove (see \cite[Section 2]{abl1}) the existence of unique positive Borel
measures $\mu_\pm$ on $\Gamma_\pm$ such that the measure $\mu$ on
$\O_\pm$ is identified to the product measure of $\mu_\pm$ with the Lebesgue measure on $\mathbb{R}$ (see \cite[Proposition 2.10]{abl1}). The construction of such measures allow to define the trace spaces
$$L^1_{\pm}=L^1(\Gamma_\pm,\d\mu_\pm)$$
with the usual norm. In the Banach space
$$X:=L^1(\O,\d\mu)$$
endowed with its usual norm, we can define the
maximal \textit{transport operator} $(\T_\mathrm{max},\D(\T_\mathrm{max}))$  as follows  (see \cite[Theorem 3.6]{abl1})
\begin{definition}\label{def:transp}
Given $f\in L^1(\O,\d\mu),$ $f$ belongs to the domain $\D(\T_\mathrm{max})$ of $\T_\mathrm{max}$ if and only if
there exists $g \in L^1(\O,\d\mu)$ and a \emph{representative} $f^\sharp$ of $f$
(i.e. $f^\sharp(x) = f(x)$ for $\mu$-a.e. $x\in\O$) such that, for $\mu$-almost every $\x \in \O$ and any $-\tau_-(\x) < t_1 \leq t_2 < \tau_+(\x)$ one has
\begin{equation}\label{fsharp}
f^\sharp(\Phi(\x,t_1)) - f^\sharp(\Phi(\x,t_2)) = \int_
{t_1}^{t_2}g(\Phi(\x,s))\d s.
\end{equation}
In this case, we set ${\T}_\mathrm{max}f = g$.
\end{definition}
\begin{remark} Notice that the above operator $\T_{\mathrm{max}}$ is \emph{well-defined}, i.e. $\T_\mathrm{max}f$ does not depend on the representative $f^\sharp$. Finally, we wish to emphasize the fact that the domain $\D(\T_\mathrm{max})$ is precisely the set of functions $f \in L^1(\O, \d\mu)$ that admit a
representative which is absolutely continuous along almost any characteristic curve.
\end{remark}

With the above definition, each function $f\in \D(\T_\mathrm{max})$ is such that the limits
\begin{equation*}
\B^+f(\y):=\lim_{s \to 0+}f^\sharp({\Phi}(\y,-s)) \qquad \text{ and } \qquad \B^-f(\y):=\lim_{s \to 0+}f^\sharp({\Phi}(\y,s))\end{equation*} exist for almost $\mu_{\pm}$-every $\y \in \Gamma_\pm$ \cite[Proposition 3.16, Definition 3.17]{abl1}. Notice that the \textit{traces} $\B^\pm f$ of a given $f\in \D(\T_\mathrm{max})$ does not necessarily belong to $L^1_\pm$.
Nevertheless one can prove the following  \cite[Theorem 3.1, Proposition 3.2, Corollary 2.1]{abl2} :
\begin{theorem}
\label{theotrace} Define the following measures over $\Gamma_{\pm}$:
$$\d\xi_{\pm}(\y)=\min\left(\t_{\mp}(\y),1\right)\d\mu_{\pm}(\y),
\qquad \y \in \Gamma_{\pm}.$$ Then, for any $f \in \D(\T_\mathrm{max})$,
the trace $\B^{\pm}f$ belongs to $Y_{\pm}:=L^1(\Gamma_{\pm},\d\xi_{\pm})$ with
$$\|\B^{\pm}f\|_{Y_{\pm}} \leq
\|f\|_X+\|\T_\mathrm{max}f\|_{X}, \qquad f \in
\D(\T_\mathrm{max}).$$
Moreover
$$\mathscr{W}:=\left\{f \in
\D(\T_\mathrm{max})\,;\,\B^-f \in \lm\right\} =\left\{f \in
\D(\T_\mathrm{max})\,;\,\B^+f \in \lp\right\} $$
 and the \textit{Green formula}
 \begin{equation}\label{greenform}\int_{\O}\T_{\mathrm{max}}f\d\mu=\int_{\Gamma_-}\B^-f\d\mu_-
-\int_{\Gamma_+}\B^+f\,\d\mu_+.\end{equation}
holds for any $f\in\mathscr{W}$.
\end{theorem}

We are then in position to define the transport operator associated to a bounded boundary operator as follows:
\begin{definition}\label{defiH} For any bounded linear operator
$$H \in \mathscr{B}(L^1_+,L^1_-)$$
we define \emph{the transport operator} $(\T_H,\D(\T_H))$ \emph{associated to the boundary condition} $H$ as:
\begin{equation}\begin{split}\label{TH}
\D(\T_H)=\{f\in \D(\T_\mathrm{max})\;;\; &\B^+f \in L^1_+  \text{ and }    \B^-f = H\B^+f \},\\
\quad {\T}_H f &= \T_\mathrm{max} f \quad \forall f\in \D(\T_H).\end{split}
\end{equation}
\end{definition}

\subsection{Construction of the semigroup associated to boundary operator with unit norm}

\label{sec3}

We begin by introducing several notations, taken from \cite{abl2}.  For any $\lambda
> 0$ one defines the following operators
\begin{equation*}
\begin{cases}
M_\lambda \::\:&Y_- \longrightarrow Y_+\\
& u \longmapsto \left[M_{\lambda}u\right](\y)=u({\Phi}(\y,-\t_-(\y)))\exp\left(-\lambda
\t_-(\y)\right)\chi_{\{\t_-(\y) < \infty\}},\:\:\:\y \in \Gamma_+\;;
\end{cases}\end{equation*}
\begin{equation*}
\begin{cases}
\Xi_{\lambda} \::\:&Y_- \longrightarrow X\\
&u \longmapsto
\left[\Xi_{\lambda}u\right](\x)=u({\Phi}(\x,-\t_-(\x)))\exp\left(-\lambda
\t_-(\x)\right)\chi_{\{\t_-(\x)<\infty\}},\:\:\:\x \in {\O}\;;
\end{cases}
\end{equation*}
\begin{equation*}
\begin{cases}
G_{\lambda} \::\:&X \longrightarrow L^1_+\\
&f \longmapsto \left[G_{\lambda}f\right]({\z})=\displaystyle
\int_0^{\t_-({\z})}f({\Phi}({\z},-s))\exp(-\lambda s)\d s,\:\:\:{\z}
\in \Gamma_+\;;\end{cases}
\end{equation*}
and
\begin{equation*}
\begin{cases}
C_{\lambda} \::\:&X \longrightarrow X\\
&f \longmapsto \left[C_{\lambda}f\right](\x)=\displaystyle
\int_0^{\t_-(\x)}f({\Phi}(\x,-s))\exp(-\lambda s)\d s,\:\:\:\x \in
{\O}
\end{cases}
\end{equation*}
where $\chi_A$
denotes the characteristic function of a set $A$.
One has the following where $\T_0$ denotes the transport operator associated to the boundary operator $H \equiv 0$:
\begin{lemma}\label{lemmaML}
For any $\lambda >0$, the following hold:
\begin{enumerate}
\item $M_\lambda \in \mathscr{B}(Y_-,Y_+)$.
Moreover, given $u \in Y_-$,  $M_\lambda u \in L^1_+$ if and only if
$u \in L^1_-$.
\item $\Xi_\lambda \in \mathscr{B}(Y_-,X)$. Moreover, the range of $\Xi_\lambda$ is a subset of $\D(\T_\mathrm{max})$ with
\begin{equation}\label{propxil1}\T_\mathrm{max}\Xi_{\lambda}u=\lambda \Xi_\lambda u, \qquad
\B^-\Xi_\lambda u=u,\:\:\B^+\Xi_\lambda u = \ml u, \qquad  \forall u
\in Y_-
\end{equation}
\item $G_\lambda \in \mathscr{B}(X,L^1_+)$. Moreover, $G_\lambda$ is surjective.
\item $C_\lambda \in \mathscr{B}(X)$ with range included in $\D(\T_\mathrm{0})$. Moreover, $C_\lambda=(\lambda-\T_0)^{-1}$ and
$$G_{\lambda}f=\B^+C_{\lambda}f \quad
\text{ for any }  f \in X .$$
\end{enumerate}
\end{lemma}
Notice that, if $H \equiv 0$, it is not difficult to check that $(\T_0,\D(\T_0))$ is the generator of a $C_0$-semigroup $\uot$ given by
\begin{equation}\label{u0}
U_0(t)f(\x)=f({\Phi}(\x,-t))\chi_{\{t <
\tau_-(\x)\}}(\x), \qquad (\x \in {\O},\:f \in X).\end{equation}

In all the sequel, we shall assume that $H: L^1_+\rightarrow L^1_-$ is a positive boundary operator of unit norm, i.e.
\begin{equation}\label{hypp}
H \in \Be(\lp,\lm)\;;\,H f \geq 0 \quad \forall f \in \lp, f \geq 0\:; \quad
\;\|H\|_{\Be(\lp,\lm)}=\sup_{\|f\|_{\lp} =1} \|H f\|_{\lm}=1.
\end{equation}
Under such an assumption, for any $0 \leq r < 1$, the boundary operator $H_r:=rH$ is dissipative with
$$\|H_r\|_{\Be(\lp,\lm)}=r < 1;$$
it is then well-known \cite[Corollary 4.1]{abl2} that the transport operator $\T_{H_r}$ generates a positive contraction semigroup $(V_r(t))_{t\geq 0}$ for any $0 \leq r < 1$. From \cite[Theorem 6.2]{abl2}, one has the following:
\begin{theorem}\label{limiting} Let $H$ satisfy Assumption \ref{hypp}. Then, for any $t \geq 0$ and any $f \in X$ the
limit $V_H(t)f=\lim_{r \nearrow 1}V_{r}(t)f$ exists in $X$ and
defines a substochastic semigroup $\vt$. If $(\A,\D(\A))$ is the
generator of $\vt$, then its resolvent is given by \begin{equation}
\label{resoA}(\lambda-\A)^{-1}f=\lim_{r \nearrow 1}(\lambda-\T_{H_r})^{-1}f=C_{\lambda}f+\sum_{n=0}^\infty
\Xi_{\lambda}H(M_{\lambda}H)^{n}G_{\lambda}f \qquad \text{ for any }
f \in X, \:\lambda >0,\end{equation} where the series converges in
$X$. Moreover, $\A$ is an extension of $\T_H$; more precisely
$$\D(\T_H) \subset \D(\A) \subset \D(\T_\mathrm{max}) \quad \text{ with } \quad \A f=\T_\mathrm{max} f \qquad \forall f \in \D(\A)$$
and
$$\D(\T_H)=\left\{\varphi \in \D(\A)\,;\,\B^+\varphi \in \lp\right\}=\left\{\varphi \in \D(\A)\,;\,\B^-\varphi \in \lm\right\}=\D(\A) \cap \mathscr{W}$$
\end{theorem}


\section{A new characterization of $(V_H(t))_{t \geq 0}$}

In this
section, we present a new characterization as well as practical expression of the semigroup $(V_H(t))_{t \geq 0}$. Indeed
in the following Theorem \ref{charaA}  we are able to prove that $(V_H(t))_{t \geq 0}$ is the smallest substochastic $C_0$-semigroup generated
by an extension of $\T_H$, while in Theorem \ref{theoDP} we show that $(V_H(t))_{t \geq 0}$  can be written as the sum of a strongly
convergent series. We first need to recall the definition of transport operator associated to an unbounded boundary operator. Precisely, let us introduce $\X$ as the space of elements
$(\psi_+,\psi_-) \in Y_+ \times Y_-$ such that $\psi_+-\ml \psi_-
\in L^1(\Gamma_+,\d\mu_+)$ for some/all $\lambda
>0.$ We equip $\X$ with the norm
\begin{equation*}\|(\psi_+,\psi_-)\|_{\X}:=\|\psi_+\|_{Y_+} + \|\psi_-\|_{Y_-}
+\|\psi_+-M_1 \psi_-\|_{\lp}\end{equation*} that makes it a Banach
space. Then, one has the following generalization of Definition \ref{defiH}:
\begin{definition}\label{def}
Given a possibly unbounded operator $\mathcal{K}$ from $Y_+$
to $Y_-$, we denote by $\D(\mathcal{K})$ its domain and $\G(\mathcal{K})$ its graph. If $\G(\K) \subset
\mathscr{E}$ we can define the transport operator $\T_{\K}$ associated to the boundary operator $\K$ by
$\T_\K f=\T_\mathrm{max}f$ for any $f \in \D(\T_\K)$, where
$$\D(\T_\K)=\bigg\{f \in \D(\T_\mathrm{max})\,;\,(\B^+f,\B^-f) \in
\G(\K)\bigg\}.$$
\end{definition}

We then have the following
\begin{lemma}\label{lem3} Let $\K$ be an unbounded operator as in  Definition \ref{def}.  For any $\lambda > 0$, the following are equivalent
\begin{enumerate}
\item $(I - M_{\lambda}\mathcal{K}): \mathscr{D}(\mathcal{K} )\mapsto L^1_+$ is bijective;
\item $(\lambda I - {\T}_{\mathcal{K}}):  \D(\T_\K)\mapsto X$ is bijective.
\end{enumerate}
 \end{lemma}
\begin{proof} According to \cite[Lemma 4.2]{abl2}, for any $\lambda > 0$ one has $[I - M_{\lambda}\mathcal{K}]\mathscr{D}(\mathcal{K} )= L^1_+$ if and only if $[\lambda I - \T_\K]\mathscr{D}({T}_{\mathcal{K}})= X$. Therefore we have only to prove that, given $\lambda > 0$,  $(I - M_{\lambda}\mathcal{K}): \mathscr{D}(\mathcal{K} )\mapsto L^1_+$ is injective if and only if  $(\lambda I - \T_\K):  \mathscr{D}(\T_\K)\mapsto X$ is injective.

Assume now that $(\lambda I - \T_\K):  \mathscr{D}(\T_\K)\mapsto X$ is injective and let $\psi \in \D(K)$ be a solution to $(I - M_{\lambda}\mathcal{K})\psi=0.$ Set $f=\Xi_\lambda K \psi$. One deduces from  Lemma \ref{lemmaML} \textit{(2)} (with $u=K\psi$) that $f \in \D(\T_\mathrm{max})$ with $\T_\mathrm{max}f=\lambda f$, $\B^-f=K\psi$ and $\B^+f=M_\lambda K\psi=\psi$. In other words, $f \in \mathscr{D}(\T_\K)$ is a solution to the equation $ (\lambda - \T_\K)f = 0$ and therefore $f=0$. Since $\psi=\B^+ f$, one gets $\psi=0$ and $(I - M_{\lambda}\mathcal{K}): \mathscr{D}(\mathcal{K} )\mapsto L^1_+$ is injective. Conversely, assume $(I - M_{\lambda}\mathcal{K}): \mathscr{D}(\mathcal{K} )\mapsto L^1_+$ to be injective and let $f \in \D(\T_\K)$ be a solution to $(\lambda-\T_\K)f=0.$ According to  \cite[Theorem 3.2]{abl2} (see also Lemma \ref{lemmaML} \textit{(2)}),  $f\in \mathscr{D}(\T_\mathrm{max})$ with $\B^+f\in \mathscr{D}(\mathcal{K})$, and $f = \Xi_{\lambda}\mathcal{K}\B^+f$. Setting then $\psi=\B^+f$, one has $\psi \in \mathscr{D}(\mathcal{K})$ and  $(I - M_{\lambda}\mathcal{K})\psi = 0$. By assumption, $\psi=0$ and, since $f=\Xi_\lambda K \psi$, $f=0$ and $(\lambda I - \T_\K):  \mathscr{D}(\T_\K)\mapsto X$ is injective.  This proves the desired equivalence.
\end{proof}

With this in hands, one can prove the following which somehow characterizes the class of operators sharing the properties of the generator $\mathcal{A}$ (recall that, according to Theorem \ref{limiting}, $\mathcal{A}$ satisfies the following properties \textit{(a)--(c)}):

\begin{proposition}\label{th3}
Let $\mathcal{A}_0$ be the generator of a strongly continuous substochastic semigroup $(V(t))_{t\geq 0}$ in $X$. Assume further that
\begin{enumerate}[(a)\:]
\item $\mathscr{D}({\T}_{H})\subseteq\mathscr{D}(\mathcal{A}_0)\subseteq\mathscr{D}({\T}_{\mathrm{max}})$

\item $\mathcal{A}_0 f = {\T}_\mathrm{max} f $ for any $f \in \mathscr{D}(\mathcal{A}_0)$

\item $\mathscr{D}({\T}_{H}) = \{f \in \mathscr{D}(\mathcal{A}_0) : \B^+f \in L^1_+\} = \{f \in \mathscr{D}(\mathcal{A}_0): \B^-f \in L^1_-\}$.
\end{enumerate}
Then there exists a boundary linear operator $\mathcal{H}_0$ from $Y_+$ to $Y_-$ with the following properties:

\begin{enumerate}[(i)\:]
\item $L^1_+ = \mathscr{D}(H) \subseteq \mathscr{D}(\mathcal{H}_0)$ with $\mathcal{H}_0\psi = H\psi$ for any $\psi \in L^1_+$

\item $\mathcal{A}_0 = {\T}_{\mathcal{H}_0}$

\item for any $\lambda > 0$ the mapping $(I - M_{\lambda}\mathcal{H}_0): \mathscr{D}(\mathcal{H}_0 )\mapsto L^1_+$ is bijective, and

\begin{equation}
(\lambda - \mathcal{A}_0)^{-1}f =  C_{\lambda}f + \Xi_{\lambda}\mathcal{H}_0(I - M_{\lambda}\mathcal{H}_0)^{-1}G_{\lambda}f.
\label{res2}
\end{equation}

\item for any $\lambda > 0$, $u\in L^1_+$, $u\geq 0$ one has

\begin{equation}
(I - M_{\lambda}\mathcal{H}_0)^{-1}u\geq 0 \quad \quad \mathcal{H}_0(I - M_{\lambda}\mathcal{H}_0)^{-1}u \geq 0
\label{res4}
\end{equation}
\end{enumerate}
\end{proposition}
\begin{proof} First of all observe that the trace mapping $\B^+:  \mathscr{D}(\mathcal{A}_0) \to Y_+$ is injective.  Indeed let $f\in \mathscr{D}(\mathcal{A}_0) $ be such that $\B^+f = 0$. Then assumption \textit{(c)} ensures that $f \in \mathscr{D}(\T_H)$, so that $\B^-f = H\B^+f = 0$.
In particular, $\|(\B^+f,\B^-f)\|_{\X} = 0$ and one deduces from  \cite[Corollary 3.1]{abl2} that  $f = 0$. Let us now introduce the set
$$E_0 :=\mathrm{Range}(\B^+\vert_{\mathscr{D}(\mathcal{A}_0)})=\{\psi\in Y_+: \exists g\in \mathscr{D}(\mathcal{A}_0) \quad \text{such that } \quad \psi = \B^+g\}$$ so that
$\B^+:  \mathscr{D}(\mathcal{A}_0) \to E_0 \subseteq Y_+$ is bijective. This allows to define an unbounded  linear boundary operator $\mathcal{H}_0: \D(\mathcal{H}_0) \to Y_-$ as follows:
\begin{equation*}
\D(\mathcal{H}_0)=E_0 \qquad \text{ and } \qquad
\mathcal{H}_0 \psi = \B^-g \qquad \forall \psi \in E_0\end{equation*}
where $g$ is the unique element of $\mathscr{D}(\mathcal{A}_0)$ such that $\B^+g = \psi$. Let us prove that $\mathcal{H}_0$ satisfies points $(i)-(iv)$.

\noindent \textit{(i)} Let $h\in L^1_+$ and $\lambda > 0$ be given. Setting $u = (I - M_{\lambda}H)h \in L^1_+$, by Lemma \ref{lemmaML} \textit{(3)}, there exists $g\in X$ such that $G_{\lambda}g = u$. Setting then$f = C_{\lambda}g + \Xi_{\lambda}Hh$ one clearly has $f \in \mathscr{D}(\T_\mathrm{max})$. Moreover $\B^+f = G_{\lambda}g + M_{\lambda}Hh = u + M_{\lambda}Hh = h$  and $\B^-f = Hh = H\B^+f $.  In other words, $f \in \mathscr{D}({T}_{H})\subseteq\mathscr{D}(\mathcal{A}_0)$. Consequently,  $h\in E_0$ with $\mathcal{H}_0h = Hh$ and \textit{(i)} is proved.

\noindent \textit{(ii)} To prove point \textit{(ii)}, it is enough to show  that $\D(\A_0)=\D(\T_{\mathcal{H}_0}).$ From the definition of $\mathcal{H}_0$ and the assumption $\D(\mathcal{A}_0) \subseteq \D(\T_\mathrm{max})$, one sees that
$$\D(\A_0) \subseteq \{f \in \D(\T_\mathrm{max})\;;\,\B^+ f \in \D(\mathcal{H}_0) \:,\:\B^-f =\mathcal{H}_0 \B^+ f\}=\D(\T_{\mathcal{H}_0}).$$
Conversely, let $f \in \D(\T_\mathrm{max})$ with $\B^+ f \in \D(\mathcal{H}_0)$ and $\B^-f =\mathcal{H}_0 \B^+ f$. By definition  of $\mathcal{H}_0$ and since $\D(\mathcal{H}_0)=E_0$, there exists $g\in \D(\A_0)$ such that $\B^+g = \B^+f$ and $\B^-g =\mathcal{H}_0 \B^+f=\B^-f$. Set $h = f - g$. One has $h\in \D(\T_\mathrm{max})$ with  $\B^+h =  \B^-h = 0$ and again, we can invoke \cite[Corollary 3.1]{abl2} to state that $h=0$, i.e. $f = g  \in \D(\mathcal{A}_0)$, proving the second inclusion.

\noindent \textit{(iii)} Since $\mathcal{A}_0$ is the generator of a substochastic semigroup we can state that for any $\lambda > 0$ and $f\in X$ there exists a unique $g \in \mathscr{D}(\mathcal{A}_0)$ such that $(\lambda - \mathcal{A}_0)g =f$, with moreover $g\geq 0$ if $f\geq 0$. This means that for any $\lambda > 0$ and $f\in X$ there exists a unique $g \in \D(\T_\mathrm{max})$, such that $\B^+g\in \D(\mathcal{H}_0)$ with $g$ solution to the boundary value problem:
\begin{equation}
(\lambda - \T_\mathrm{max}) g = f \quad \quad \B^-g = \mathcal{H}_0\B^+g.
\label{bvp}
\end{equation}
From \cite[Theorem 3.2]{abl2}, such a solution $g$ is given by
\begin{equation}
g = C_{\lambda}f + \Xi_{\lambda}\B^-g = C_{\lambda}f + \Xi_{\lambda}\mathcal{H}_0\B^+g,
\label{sobvp}
\end{equation}
and, in particular, $u:=\B^+g\in \mathscr{D}(\mathcal{H}_0)$ satisfies $(I - M_{\lambda}\mathcal{H}_0)u = G_{\lambda}f.$ Since  $(\lambda I - {\T}_{\mathcal{H}_0}):  \D(\T_{\mathcal{H}_0})\to X$ is bijective, one deduces from Lemma \ref{lem3} that  $(I - M_{\lambda}\mathcal{H}_0): \mathscr{D}(\mathcal{H}_0 ) \to L^1_+$ is bijective. Then, $u=\B^+g = (I - M_{\lambda}\mathcal{H}_0)^{-1}G_{\lambda}f$ which, from \eqref{sobvp}, shows that the solution to \eqref {bvp} becomes
$$g= C_{\lambda}f + \Xi_{\lambda}\mathcal{H}_0(I - M_{\lambda}\mathcal{H}_0)^{-1}G_{\lambda}f$$
which is nothing but \eqref{res2}.

\noindent \textit{(iv)} Let now  $\lambda > 0$ and $u\in L^1_+$ with $u\geq 0$ be given. Consider then the function $g_{\lambda}$ defined  as follows:
\begin{equation*}
g_{\lambda}(\x) = \begin{cases}
\dfrac{(1+\lambda)\tau_-(\x) + 1}{\tau_-(\x) + \tau_+(\x)}\exp(-\tau_+(\x))\,u(\Phi(\x,\tau_+(\x)) \qquad &\text{ if } \quad \tau_-(\x) + \tau_+(\x) < \infty,\\
(1+\lambda)\exp(-\tau_+(\x)) u(\Phi(\x,\tau_+(\x))  \qquad &\text{ if } \quad  \tau_-(\x)=\infty \quad \text{ and } \tau_+(\x) < \infty,\\
0 \qquad &\text{ if } \quad \tau_+(\x)=\infty.
\end{cases}
\end{equation*}
One can check easily that $g_\lambda \in X$, $g_\lambda \geq 0$ with $G_{\lambda}g_{\lambda} = u$.  Setting now
$$f_{\lambda} = (\lambda - \mathcal{A}_0)^{-1}g_{\lambda} =  C_{\lambda}g_{\lambda}+ \Xi_{\lambda}\mathcal{H}_0(I - M_{\lambda}\mathcal{H}_0)^{-1}G_{\lambda}g_{\lambda}$$
one sees that $f_\lambda$ is nonnegative, with $\B^+f_{\lambda} = (I - M_{\lambda}\mathcal{H}_0)^{-1}G_{\lambda}g_{\lambda} = (I - M_{\lambda}\mathcal{H}_0)^{-1}u \geq 0$; $\B^- f_{\lambda} =\mathcal{H}_0(I - M_{\lambda}\mathcal{H}_0)^{-1}G_{\lambda}g_{\lambda} = \mathcal{H}_0(I - M_{\lambda}\mathcal{H}_0)^{-1}u \geq 0$ which proves the result.\end{proof}

The above Proposition allows to prove that $(V_H(t))_{t\geq 0}$ is the \emph{smallest}  substochastic semigroup generated by an extension of $\T_H$. More precisely we have
\begin{theorem}\label{charaA}
Let $(V (t))_{t\geq 0}$  be a strongly continuous substochastic semigroup in $X$ with generator $\mathcal{A}_0$  which satisfies the conditions \text{(a)-(c)} of Proposition \ref{th3}. Then, for any $t\geq 0$ one has $V(t)\geq V_H(t)$, i.e. $V (t)f \geq V_H(t)f$ for any \emph{nonnegative} $f \in X$. In other words, $(V_H(t))_{t \geq 0}$ is the \emph{smallest substochastic semigroup generated by an extension of $\T_H$}.
\label{Vext}
\end{theorem}
\begin{proof} According to the previous Proposition \ref{th3}, there exists an extension $\mathcal{H}_0$ of $H$   so that the generator $\mathcal{A}_0$  of the semigroup $(V(t))_{t\geq 0}$ coincides with the transport operator ${T}_{\mathcal{H}_0}$, and formula \eqref{res2} holds. Now, since $\mathcal{H}_0h = H h$ for any $h\in L^1_+$, we have, for $0 < r < 1$ and $H_r=rH$:
\begin{multline*}
(I - M_{\lambda}\mathcal{H}_0)^{-1} - (I - M_{\lambda}H_r )^{-1} = \big[(I - M_{\lambda}\mathcal{H}_0)^{-1}(I - M_{\lambda}H_r) - I\big](I - M_{\lambda}H_r)^{-1} \\
= (I - M_{\lambda}\mathcal{H}_0)^{-1}(I - M_{\lambda}H_r  - I + M_{\lambda}\mathcal{H}_0)(I - M_{\lambda}H_r )^{-1} \\
= (1 - r)(I - M_{\lambda}\mathcal{H}_0)^{-1}M_{\lambda}H(I - M_{\lambda}H_r )^{-1}
\end{multline*}
where we used that the range of $(I-M_\lambda H_r)^{-1}$ is $\lp$. One deduces easily from this that
\begin{multline*}
\mathcal{H}_0(I - M_{\lambda}\mathcal{H}_0)^{-1} - H_r(I - M_{\lambda}H_r)^{-1}\\
= (1 - r)\left(\mathcal{H}_0(I - M_{\lambda}\mathcal{H}_0)^{-1}M_{\lambda}H(I - M_{\lambda}H_r)^{-1} + H(I - M_{\lambda}H_r)^{-1}\right).\end{multline*}
Recalling that $(\lambda-\T_{H_r})^{-1}f=C_{\lambda}f + \Xi_{\lambda}H_r(I - M_{\lambda}H_r)^{-1}G_{\lambda}f$ (see \cite[Eq. (4.6)]{abl2}), by virtue of  \eqref{res2} one has then, for any $f\in X$,
\begin{multline*}
(\lambda - \mathcal{A}_0)^{-1}f -  (\lambda - \T_{H_r})^{-1}f =\Xi_\lambda \mathcal{H}_0(I - M_{\lambda}\mathcal{H}_0)^{-1} G_{\lambda}f - \Xi_\lambda H_r(I - M_{\lambda}H_r)^{-1}G_{\lambda}f  \\
= (1 - r)\Xi_\lambda \left(\mathcal{H}_0(I - M_{\lambda}\mathcal{H}_0)^{-1}M_{\lambda}H(I - M_{\lambda}H_r)^{-1} + H(I - M_{\lambda}H_r)^{-1}\right)G_{\lambda}f.\end{multline*}
If $f\geq 0$, according to Proposition \ref{th3} \textit{(iv)}, we get $(\lambda - \mathcal{A}_0)^{-1}f \geq (\lambda - {\T}_{H_r})^{-1}f$ for any $0< r < 1$.  This inequality together with \eqref{resoA}  allow to state that $(\lambda - \mathcal{A}_0)^{-1} \geq (\lambda - \mathcal{A})^{-1}$ which gives the result according to  the
exponential formula. \end{proof}

We recall now the recent result of the first author \cite{al3} about the construction of a suitable strongly continuous family of bounded linear operators in $X$.
First, let
$$\mathcal{D}_0=\{f \in \mathscr{D}(\T_\mathrm{max})\::\:\B^{\pm} f=0\}.$$
The subset $\mathcal{D}_0$ is dense in $X$ (see \cite[Proposition 1]{al3}). Remember that the semigroup $\uot$ is defined through \eqref{u0}. Now, one introduces the following
\begin{definition}\label{UK}
For any $t \geq 0$, we define the family $(U_k(t))_{k \in \mathbb{N}}$ by induction as follows: if $f \in \mathcal{D}_0$, $t > 0$ and $k \geq 1$, one sets
\begin{equation}\label{Uk(t)}
U_k(t)f(\x)=\begin{cases}
 H(\B^+U_{k-1}(t - \tau_-(\x))f)(\Phi(\x,-\tau_-(\x)))&
\forall \x\in \O  , \quad with \quad \tau_-(\x)\leq t,\\
0 & \forall \x\in \O  \quad with \quad \tau_-(\x) > t.
\end{cases}\end{equation}
Moreover, for $t=0$, we set $U_k(0)f = 0$ for any $k \geq 1$ and any $f \in X$.
\end{definition}
\begin{remark} In other words, if we put $\O_t := \{\x\in\O: \x = \Phi(\y,s), \y\in\Gamma_-, 0 < s < t\wedge\tau_+(\y)\}$, then $[U_k(t)f](\x)$ may be different from zero only for $\x \in\O_t$, being $U_k(t)f(\Phi(\y,s)) = H(\B^+U_{k-1}(t - s)f)(\y)$.
\end{remark}
\begin{remark}\label{rem:B+z}
Notice that, given $f \in \mathcal{D}_0$ and $t > 0$, one has $\left(\int_0^t U_k(s)f \d s\right)(\x)=0$  for any $\x \in \O$ with $\tau_-(\x) > t $. In particular,
\begin{equation}\label{eq:B^+z}
\B^+\left(\int_0^t U_k(s)f \d s \right)(\z)=0 \qquad \forall \z \in \Gamma_+\,;\;\tau_-(\z) > t\:,\;k \geq 1.\end{equation}
\end{remark}
The properties of the family $(U_k(t))_{t \geq 0}$, for given $k \geq 1$, have been established in \cite{al3}. In particular, for any $f \in \mathcal{D}_0$ and any $t > 0$, one has $U_k(t)f \in X$ with
$$\|U_k(t)f \|_X \leq \|H\|_{\Be(\lp,\lm)}^k \|f\|_X=\|f\|_X \qquad \forall k \geq 1.$$
Since $\mathcal{D}_0$ is dense in $X$, one can extend $U_k(t)$ in a bounded linear operator in $X$, still denoted $U_k(t)$ such that
$$\|U_k(t)\|_{\Be(X)} \leq 1.$$
Moreover, one has the following
\begin{proposition}\label{propertiesUk}
For any $k \geq 1$, the family $(U_k(t))_{t\geq 0}$ enjoys the following properties:
\begin{enumerate}
\item $(U_k(t))_{t\geq 0}$  is a strongly continuous family of operators in $X.$
\item For all $f\in \mathcal{D}_0$ and $t \geq 0$ one has $U_k(t)f \in \D(\T_\mathrm{max})$ with $\T_\mathrm{max}U_k(t)f = U_k(t){\T}_{\mathrm{max}}f$.
\item For all $f\in  \mathcal{D}_0$ and $t \geq 0$ the traces $\B^\pm U_k(t)f \in L^1_\pm$ and the mappings $t \mapsto \B^\pm U_k(t)f \in L^1_\pm$ are continuous.
\item For any $f\in X$, $t \geq 0$ and  $s \geq 0$ we have $U_k(t +
s)f = \sum_{j=0}^kU_j(t)U_{k-j}(s)f$.
\item For all $f\in X$ and $t > 0$ one has $\int_0^t U_k(s)f\d s\in
 \D(\T_\mathrm{max})$ with
 $$ \T_\mathrm{max} \int_0^t U_k(s)f\d s = U_k(t)f.$$
 Moreover, $\B^\pm \left(\int_0^tU_k(s)f \d s\right) \in L^1_\pm$ and
 \begin{equation}\label{B^+int}
 H \B^+\left(\int_0^t U_{k-1}(s)f \d s \right) =\B^-\left(\int_0^t U_k(s)f \d s\right).
 \end{equation}
 \item For any $f\in X$ and $\lambda > 0$, setting $g_k : = \int_0^{\infty}\exp(-\lambda t)U_k(t)f\d t$, one has $g_k \in \D(\T_\mathrm{max})$ with
 \begin{equation*}
  \T_\mathrm{max}g_k =\lambda g_k \quad \text{ for } \quad k \geq 1, \qquad \text{ while } \quad
   \T_\mathrm{max}g_0=\lambda g_0 - f ;\end{equation*}
   and $\B^+g_k = (M_{\lambda}H)^k G_{\lambda}f\in L^1_+$ for any $k\geq 0$ while  $\B^-g_0 =  0$ and $\B^-g_k = H\B^+g_{k - 1}$ if $k\geq 1.$
\item For any nonnegative $f \in X$ and any $t \geq 0$ and $n \geq 1$ one has
\begin{equation}\begin{split}
\sum_{k=0}^{n}\|U_k(t)f\|_X &= \|f\|_X - \left\|\B^+\int_0^t U_n(s)f\d s\right\|_{L^1_+}\\
&+\sum_{k=0}^{n-1}\left[\left\|H\B^+\int_0^t U_k(s)f\d s\right\|_{L^1_-} - \left\|\B^+\int_0^t U_k(s)f\d s\right\|_{L^1_+}\right] .
\label{sumnorm}
\end{split}\end{equation}
In particular,
\begin{equation}\label{bounded}\sum_{k=0}^{n}\|U_k(t)f\|_X \leq \|f\|_X - \left\|\B^+\int_0^t U_n(s)f\d s\right\|_{L^1_+} \leq \|f\|_X.\end{equation}
\end{enumerate}
\end{proposition}

The above listed properties allow  to give a characterization of the semigroup  $(V_H(t))_{t\geq 0}$ in terms of a strongly convergent expansion series, reminiscent to classical Dyson-Phillips expansion series for additive perturbation:
\begin{theorem}\label{theoDP}
For any $f \in X$ and any $t \geq 0$, one has
\begin{equation}
V_H(t)f = \sum_{k=0}^{\infty}U_k(t)f .
\label{sum}
\end{equation}
\label{thuv}
\end{theorem}
\begin{proof} For any  $f \in X$ and any $t \geq 0$, set $V(t)f=\sum_{k=0}^{\infty}U_k(t)f.$ Notice that the  series is convergent in $X$ and  the family $(V(t))_{t\geq 0}$ defines a substochastic $C_0$-semigroup in $X$ (see \cite[Theorem 4.3]{al3} for details). Let us prove that $V(t)=V_H(t)$  for all $t\geq 0$. Let $f \in X$ and $\lambda > 0$ be fixed. Set, for any $k\geq 1$,
$$g_k = \int_0^{\infty}\exp(-\lambda t)U_k(t)f\d t.$$
 Proposition \ref{propertiesUk} asserts that $g_k \in \D(\T_\mathrm{max})$ and satisfies $\T_\mathrm{max}g_k=\lambda g_k$ for any $k \geq 1.$ According to \cite[Theorem 2.1]{abl2} we deduce that, for $k \geq 1$, $ g_k = \Xi_{\lambda}H\B^+g_{k-1} = \Xi_{\lambda}H(M_{\lambda}H)^{k-1}G_{\lambda}f$. Summing this identity, we get that
$$\int_0^{\infty}\exp(-\lambda t)W(t)f\d t = \sum_{k=0}^{\infty}g_k  = C_{\lambda}f + \sum_{k=0}^{\infty}\Xi_{\lambda}H(M_{\lambda}H)^{k}G_{\lambda}f.$$
Since this last expression coincides with $(\lambda-\A)^{-1}f$, one deduces from the injectivity of Laplace transform that $V(t)f=V_H(t)f$ for any $t\geq0.$\end{proof}

An immediate consequence of the above Theorem \ref{thuv} is given in the following
\begin{corollary}
For any $f\in X$ and $\lambda > 0$,
as $n \rightarrow \infty$, the sum $\displaystyle \sum_{k = 0}^{n}\int_0^{\infty} \exp(-\lambda t)U_k(t)f\d t$ converges to  $(\lambda - \mathcal{A})^{-1}f$ in the graph norm of $\mathcal{A}$.
\label{coruv}
\end{corollary}

We end this section with a technical result that complements Proposition \ref{propertiesUk} and shall be useful in the sequel
\begin{lemma}\label{lemmeHBint} Let $f \in X$ be nonnegative and $t > 0$ be given. For any $\z\in \Gamma_+$ and any $k \geq 1$ it holds
$$\left[ \B^+\int_0^t U_k(s)f \d s \right](\z) \leq  \left[H\B^+\int_0^t U_{k-1}(s)f\d s\right](\Phi(\z,-\tau_-(\z)).$$
\end{lemma}
\begin{proof} Let $k\geq 1$ and $\z \in \Gamma_+$ be given. If $\tau_-(\z) > t$, one gets from \eqref{eq:B^+z} that
$$\B^+\left(\int_0^t U_k(s)f\d s\right)(\z)=0$$
from which the conclusion clearly holds. Now, if $\tau_-(\z) \leq t$, set $\y =\Phi(\z,-\tau_-(\z)) \in \Gamma_-$. Since  $\int_0^t U_k(s)f\d s\in
 \D(\T_\mathrm{max})$ with  $\T_\mathrm{max} \int_0^t U_k(s)f\d s = U_k(t) f$, one deduces from Definition \ref{def:transp} (see also \cite[Theorem 3.6]{abl1}) that
 $$\int_{t_1}^{t_2} [U_k(t)f](\Phi(\y,s))\d s = \left[\int_0^t U_k(s)f\d s\right](\Phi(\y,t_1)) - \left[\int_0^t U_k(s)f\d s\right](\Phi(\y,t_2))$$
 for any $0 < t_1 < t_2 < \tau_+(\y)=\tau_-(\z)  \leq t.$ In particular, for nonnegative $f$ we get
$$\left[\int_0^t U_k(s)f\d s\right](\Phi(\y,t_2)) \leq  \left[\int_0^t U_k(s)f\d s\right](\Phi(\y,t_1)) \qquad \forall 0 < t_1 < t_2 < \tau_+(\y)=\tau_-(\z)  \leq t.$$
Letting $t_1 \to 0^+$ and $t_2 \to \tau_+(\y)$ and since $\z=\Phi(\y,\tau_+(\y))$, we get
$$\left[\B^+  \int_0^t U_k(s)f \d s \right](\z) \leq \left[\B^-\left(\int_0^t U_k(s)f\d s\right)\right](\y).$$
Using now \eqref{B^+int} and the fact that $\y=\Phi(\z,-\tau_-(\z))$ we get the conclusion.
\end{proof}
\section{Honesty theory}
\def \cc {\mathfrak{a}}
\subsection{On some functionals}
\label{sec4}
\setcounter{equation}{0}

For any $f\in \D(\T_\mathrm{max})$ we define
$${\cc}(f) = - \int_{\O} {\T}_\mathrm{max} f d\mu.$$
while, for any $f \in \mathscr{W}$, we set
$$\cc_0(f) = \int_{\Gamma_+} \B^+f d\mu_+ - \int_{\Gamma_-} H\B^+f d\mu_-.$$
Clearly ${\cc}: \D(\T_\mathrm{max}) \rightarrow \mathbb{R}$ is a linear functional with $|{\cc}(f) | \leq \|{\T}_\mathrm{max} f\|_X$ for any $f\in \D(\T_\mathrm{max})$. Here we are interested in the restriction of ${\cc}$ to $\D(\A)$, that we still denote by ${\cc}$. Since  $\mathcal{A}$ generates a positive contraction semigroup $(V_H(t))_{t\geq 0}$ we have
$${\cc}(f)=-\int_{\O} \A f \d \mu= \lim_{t \to 0^+} t^{-1}\int_{\O} \left(f-V_H(t)f  \right)\d\mu \geq 0 \qquad  \forall f\in {\D(\A)_+}: = {\D(\A)}\cap X_+.$$
Hence ${\cc}\::\:\D(\A) \rightarrow \mathbb{R}$ is a positive linear functional. Furthermore ${\cc}$ is continuous in the graph norm of $\mathcal{A}$ and its restriction to $\mathscr{D}(\T_H)$ is equal to the restriction of $\cc_0$ to $\mathscr{D}(\T_H)$. Indeed, according to Green's formula \eqref{greenform} for all $f\in \mathscr{D}({\T}_H)$ we have
$${\cc}(f) = \int_{\Gamma_+} \B^+f d\mu_+ - \int_{\Gamma_-} \B^-f d\mu_- = \int_{\Gamma_+} \B^+f d\mu_+ - \int_{\Gamma_-} H\B^+f d\mu_- = \cc_0(f).$$

This basic observation allows to formulate an equivalent to \cite[Proposition 4.5]{almm} in this boundary perturbation context. Precisely, one has
\begin{proposition}\label{prop4}
For all $f\in {\D(\A)}$ there exists
\begin{equation}
\lim\limits_{t\rightarrow 0+}\frac{1}{t} \sum_{k = 0}^{\infty} \cc_0\left(\int_0^t U_k(s)f \d s\right) =: \widehat{\cc}(f)
\label{fun2}
\end{equation}
with $|\widehat{\cc}(f)| \leq 2 (\|f\|_X + \|\A f\|_X)$. Furthermore, if $f\in \D(\A)_+$, then
\begin{equation}
0\leq \widehat{\cc}(f) \leq \cc(f) \leq \|\T_\mathrm{max} f\|.
\label{disfun2}
\end{equation}
\end{proposition}
The proof of Proposition \ref{prop4} is based upon the following
\begin{lemma}\label{lemme42}
For any $f\in X$ and $t > 0$ one has
\begin{equation}
\left|\sum_{k = 0}^{\infty} \cc_0\left(\int_0^t U_k(s)f \d s\right)\right| \leq \|f\|_X.
\label{stim2}
\end{equation}
If $f\in {\D(\A)}$ then one also has
\begin{equation}
\left|\sum_{k = 0}^{\infty}  \cc_0\left(\int_0^t U_k(s)f \d s\right)\right| \leq 2 t\left(\|f\|_X + \|\T_\mathrm{max} f\|_X\right).
\label{stima3}
\end{equation}
\end{lemma}
\begin{proof} For simplicity, for any \emph{fixed} $t > 0$, we set
$$G_k(f)=\int_0^t U_k(s) f \d s \qquad \forall k \geq 1.$$
According to Proposition \ref{propertiesUk} (5),  $G_k(f) \in \D(\T_\mathrm{max})$ for any $f\in X$, $k\geq 1$   with moreover
$\B^+G_k(f) \in L^1_+$, i.e. $G_k(f) \in \mathscr{W}.$  We begin with assuming  $f\in X_+$ and $t > 0$. One can reformulate \eqref{sumnorm} as
\begin{equation}
\sum_{k = 0}^{n-1} \cc_0\left(G_k(f)\right) = \|f\|_X - \sum_{k = 0}^{n}\|U_k(t)f\|_X - \| \B^+G_n(f)\|_{L^1+} \leq \|f\|_X
\label{stim0}
\end{equation}
Therefore, we can see that $\left(\sum_{k = 0}^n \cc_0\left(G_k(f)\right)\right)_n$ is an increasing bounded sequence whose limit satisfies
\begin{equation}\label{stima1}
\sum_{k = 0}^{\infty} \cc_0\left(G_k(f)\right) \leq \|f\|_X - \sum_{k = 0}^{\infty}\|U_k(t)f\|_X.\end{equation}
Now, for general $f \in X$, since $G_k(f) \in \mathscr{W}$, we deduce from \cite[Proposition 2.2]{abl2} that $|G_k(f)| \in \mathscr{W}$ and, since  $U_k(s)$ ($0 < s < t$, $k \geq 0$) is a positive operator, the inequalities
 $$\left|\cc_0\left(G_k(f)  \right)\right| \leq \cc_0\left(\left|G_k(f)\right|\right) \leq \cc_0\left( G_k(|f|)\right) \qquad \forall k \geq 1$$
 hold. This, together with \eqref{stima1} yields \eqref{stim2}. Before proving \eqref{stima3}, one notices that  the right-hand side of \eqref{stima1} for $f \geq 0$ is $$\int_\O \left(f-\sum_{k=0}^\infty U_k(t)f\right)\d \mu=\int_\O \left(f - V_H(t)f\right)\d \mu
=-\int_\O \A\left(\int_0^t V_H(s)f\d s\right) \d \mu$$
where we used Theorem \ref{theoDP} and the well-know fact (see \cite[Lemma 1.3, p. 50]{Engel}) that, for any $C_0$-semigroup $(V_H(t)_{t \geq 0}$ with generator $\A$, one has $\int_0^t V_H(s)f \d s \in \D(\A)$ with $\A \left(\int_0^t V_H(s)f\d s\right)=V_H(t)f-f$ for any $t \geq 0$ and any $f \in X$. Since moreover $V_H(t)f-f=\int_0^t V_H(s)\A f \d s$ if $f \in \D(\A)$, one gets
\begin{equation}\label{stima11}
\sum_{k = 0}^{\infty} \cc_0\left(\int_0^t U_k(s) f \d s\right) \leq \cc\left(\int_0^t V_H(s)f \d s\right)=-\int_\O  \left(\int_0^t V_H(s)\A f\d s\right) \d \mu \qquad \forall f \in \D(\A) \cap X_+.\end{equation}
Let us now  fix $f\in \D(\A)$ and set $g : = (I - \A)f = g^+ - g^-$, where $g_+$ and $g_-$ denote  respectively the positive and negative parts of $g$. Put also $f^\pm_1 = (I - \mathcal{A})^{-1}g_\pm$ so that $f = f^+_1 - {f}^-_1$, where ${f}^\pm_1$ are belonging to $\D(\A)_+$ (notice that $f_1^\pm$ do not necessarily coincide with the positive and negative parts $f^\pm$ of $f$). One has
$$\|\mathcal{A}{f}^\pm_1\|_X \leq \|{f}^\pm_1\|_X + \|g^\pm\|_X \leq 2\|g^\pm\|_X.$$
Recalling that $\mathcal{A}{f}^\pm_1 = \T_\mathrm{max}{f}^\pm_1$ and using formula \eqref{stima11}  we get
\begin{equation*}\begin{split}
\sum_{k = 0}^{\infty} \cc_0\left(\int_0^t U_k(s) f_1^\pm \d s\right) &\leq  -\int_\O  \left(\int_0^t V_H(s)\A f^\pm_1 \d s\right) \d \mu \\
&\leq \int_0^t \|V_H(s)\T_\mathrm{max}f_1^\pm\|\d s  \leq t\|\T_\mathrm{max} f_1^\pm\|_X \leq 2t\|g_\pm\|_X\end{split}\end{equation*}
where we used that the semigroup $(V_H(t))_{t \geq 0}$ is substochastic. Finally, noticing that
$$\left|\sum_{k = 0}^{\infty} \cc_0\left(\int_0^t U_k(s) f  \d s\right)\right| \leq \sum_{k = 0}^{\infty} \cc_0\left(\int_0^t U_k(s) f_1^+ \d s\right) + \sum_{k = 0}^{\infty} \cc_0\left(\int_0^t U_k(s) f_1^- \d s\right)$$ we obtain \eqref{stima3}  since $\|g^+\|_X + \|g^-\|_X = \|g\|_X \leq \|f\|_X + \|\T_\mathrm{max} f\|_X$.\end{proof}

\begin{proof}[Proof of Proposition \ref{prop4}] Using Lemma \ref{lemme42} together with a repeated use of Proposition \ref{propertiesUk} \textit{(4)}, it is not difficult to resume the proof of \cite[Proposition 4.5]{almm} to get the result. We only mention here that the equivalent of \cite[Eq. (4.14)]{almm} in our context is
 \begin{equation}
\sum_{k=0}^{\infty}\cc_0\left(\int_0^t U_k(s)f ds\right) = \lim_{\tau\rightarrow 0+}\frac{1}{\tau}\sum_{k=0}^{\infty}\cc_0\left(\int_0^{\tau} U_k(r)\left(\int_0^tV_H(s)f\d s\right)\d r\right).
\label{lim2}
\end{equation}
Details are omitted.
\end{proof}

As an immediate consequence of Proposition \ref{prop4} we deduce the following
\begin{corollary}\label{corol1}
 For any $f\in X$, $t > 0$ and $\lambda > 0$ one has
 \begin{equation}
\widehat{\cc}\left(\int_0^t V_H(s)f \d s\right) = \sum_{k = 0}^{\infty} \cc_0\left(\int_0^t U_k(s)f \d s\right),
  \label{bart}
  \end{equation}
  and
\begin{equation}
\widehat{\cc}\left((\lambda - \mathcal{A})^{-1}f\right) = \sum_{k = 0}^{\infty}\left(\int_{\Gamma_+}(M_ {\lambda}H)^k G_{\lambda}f \d\mu_+ - \int_{\Gamma_-}H(M_ {\lambda}H)^kG_{\lambda}f \d\mu_-\right).
  \label{barla}
  \end{equation}\end{corollary}
  \begin{proof} Identity \eqref{bart} is simply deduced from \eqref{lim2} and the definition \eqref{fun2}. Regarding \eqref{barla}, observe that for any $f\in X$, and $\lambda > 0$ one has
$$(\lambda - \mathcal{A})^{-1}f = \int_0^{\infty}\exp(-\lambda t)V_H(t)f\d t = \lambda \int_0^{\infty}\exp(-\lambda t)\left(\int_0^t V_H(s)f\d s\right)\d t .$$
Therefore, from \eqref{bart},
\begin{equation*}\begin{split}
\widehat{\cc}\left((\lambda - \mathcal{A})^{-1}f\right)&= \lambda \int_0^{\infty}\exp(-\lambda t)\widehat{\cc}\left(\int_0^t V_H(s)f \d s\right)\d t \\
&= \lambda \int_0^{\infty}\exp(-\lambda t)\sum_{k = 0}^{\infty} \cc_0\left(\int_0^t U_k(s)f \d s\right)\d t.\end{split}\end{equation*}
Setting, $g_k=\displaystyle \int_0^{\infty}\exp(-\lambda t)U_k(t)f \d t$ and $\phi_k(t)=\displaystyle\int_0^t U_k(s)f \d s$, one deduces from  Proposition \ref{propertiesUk} \textit{(6)}  that, for any $k \geq 1$
$$\lambda \int_0^{\infty}\exp(-\lambda t) \B^+\phi_k(t) \d t = \B^+\ g_k= (M_ {\lambda}H)^kG_{\lambda}f,$$
and, recalling that $\cc_0(\phi_k(t))=\displaystyle \int_{\Gamma_+} \B^+\phi_k(t) \d\mu_+ - \int_{\Gamma_-} H\B^+\phi_k(t) \d\mu_-$ we get \eqref{barla}.\end{proof}
\begin{remark} In the free-streaming context, the identity \eqref{barla}  shows that the functional $\widehat{\cc}$ coincides with the functional $c_{\lambda}$ defined in \cite{mmk}. In particular, this shows that  the functional $c_{\lambda}$ of \cite{mmk}  \emph{does not depend on} $\lambda$, answering the question left open in \cite[Remark 17]{mmk}. Moreover, by Proposition \ref{prop4}, we see that the functionals $c_{\lambda}$ and $\widehat{c}$ of \cite{mmk} (corresponding respectively to our $\widehat{\cc}$ and $\cc$) are positive functionals such that $c_{\lambda}(\varphi) \leq \widehat{c}(\varphi)$ for all $\varphi\in {\D(\A)_+}$ which extends the result of \cite[Remark 17]{mmk} valid only for $\varphi\in (\lambda -\mathcal{A})^{-1}X_+$.
\label{remmu}
\end{remark}
Proposition \ref{prop4} allows to define a third linear positive functional
$\Theta: \D(\A) \to \mathbb{R}$ by setting
$$\Theta(f) = \cc(f) - \widehat{\cc}(f) \qquad \text{ for any } f\in \D(\A).$$
Clearly, the functional $\Theta$ is continuous in the graph norm of $\D(\A)$. Other properties of $\Theta$ are stated here below.
\begin{corollary}
 For any $f\in X$, $t > 0$ and $\lambda > 0$ one has
 \begin{equation}
\Theta\left(\int_0^t V_H(s)f \d s\right)=\lim_{n\rightarrow\infty}\int_{\Gamma_+}\B^+\left(\int_0^t U_n(s)f \d s\right)\d\mu_+,
  \label{bart2}
  \end{equation}
  and
\begin{equation}
\Theta\left((\lambda - \mathcal{A})^{-1}f\right)  = \lim_{n\rightarrow\infty} \int_{\Gamma_+}(M_ {\lambda}H)^nG_{\lambda}f \d\mu_+.
 \label{barl2}
 \end{equation}
In particular, both the limits appearing in \eqref{bart2} and \eqref{barl2} exist and are finite for any $f \in X.$
 \end{corollary}
\begin{proof} As in the  proof of Corollary \ref{corol1}, for fixed $f \in X$, $\lambda > 0$ and $t \geq 0$ set
$$g_k= \int_0^{\infty}\exp(-\lambda t)U_k(t)f \d t \qquad \text{ and  } \qquad\phi_k(t)=\int_0^t U_k(s)f \d s, \qquad k \geq 0.$$
Notices that $\phi_k(t) \in \mathscr{W}$ for any $t > 0$ and any $k \geq 1$. One checks then easily thanks to Proposition \ref{propertiesUk} \textit{(7)} that, for any $n \geq 1$
 $$\sum_{k = 0}^{n}\cc\left(\phi_k(t)\right) = \sum_{k = 0}^{n-1} \cc_0\left(\phi_k(t)\right) + \int_{\Gamma_+}\B^+ \phi_n(s)\d\mu_+.$$
One deduces easily \eqref{bart2} from this last identity combined with \eqref{bart} and the fact that $\left(\sum_{k = 0}^{n}\phi_k(t)\right)_n$ converges to $\int_0^t V_H(s)f \d s$ in the graph norm of $\mathcal{A}$. In the same way, noticing that for any  $n\in\mathbb{ N}$ one has
$$\sum_{k = 0}^{n}\cc\left(g_k\right) = \sum_{k = 0}^{n-1} \cc_0\left(g_k\right) + \int_{\Gamma_+}\B^+g_n\d\mu_+,$$
one readily gets \eqref{barl2} using now \eqref{barla}  together with the fact that $\left(\sum_{k = 0}^{n}g_k\right)_n$ converges to $(\lambda - \mathcal{A})^{-1}f$ in the graph norm of $\mathcal{A}$ as $n \to \infty$ (see Corollary \ref{coruv}).
\end{proof}

  The above results yield the following

  \begin{proposition}  For any $f \in \D(\T_H)$ one has $\widehat{\cc}(f) = \cc(f)= \cc_0(f)$. Consequently,
\begin{equation}\label{corhon}\Theta(f)=0 \qquad \forall f \in \D\left(\,\overline{\T_H}\,\right).\end{equation}
\end{proposition}
\begin{proof} By definition and since $\Theta$ is continuous over $\D(\A)$ endowed with the graph norm, it is enough to prove that $\widehat{\cc}(f)=\cc(f)$ for any $f \in \D(\T_H)$. For any $\lambda > 0$, since the operator $G_{\lambda}: X \rightarrow L^1_+$ is surjective, one deduces from \eqref{barl2} that the limit $\lim_{n\to\infty}\int_{\Gamma_+} \left(M_\lambda H\right)^n h \d\mu_+$ exists and is finite for any $h\in L^1_+$. Now, given $f \in \D(\T_H)$, set $g = (\lambda - \mathcal{A})f$. Since $\B^+f \in L^1_+$ the limit  $\lim_{n\to\infty} \int_{\Gamma_+}(M_ {\lambda}H)^n \B^+f \d\mu_+$ exists and is finite while, from $f = (\lambda - \mathcal{A})^{-1}g$  one deduces that $$\sum_{k = 0}^{n-1}(M_ {\lambda}H)^kG_{\lambda}g = \B^+f - (M_{\lambda}H)^n\B_+f.$$ Therefore, the sequence $\displaystyle\left(\sum_{k = 0}^{n-1}\int_{\Gamma_+}(M_ {\lambda}H)^kG_{\lambda}g \d\mu_+\right)_{n}$ is converging. In particular,
$$\lim_{n\to \infty}\int_{\Gamma_+}(M_ {\lambda}H)^nG_{\lambda}g \d\mu_+ = 0.$$ From \eqref{barl2}, this limit coincides with $\Theta\left((\lambda-\A)^{-1}g\right)=\Theta(f)$ which shows the result.\end{proof}

Now one proves that,  somehow, \eqref{corhon} is a characterization of $\D(\overline{\T_H})$, at least for nonnegative $f$ :
 \begin{proposition}
If $f \in \D(\A)_+$ is such that $\Theta(f)=0$, then, $f \in \D(\,\overline{\T_H}\,).$
  \label{prophon}
\end{proposition}
 \begin{proof} The proof is inspired by the analogous result for additive perturbation \cite[Proposition 1.6]{mmv}. Let $f \in \D(\A)_+$ be given such that $\Theta(f)=0$, i.e. $\widehat{\cc}(f)=\cc(f)$. Since $\lambda(\lambda-\mathcal{A})^{-1}f \to f$ in the graph norm of $\D(\A)$ as $\lambda \to \infty$, we get that
 $$\lim_{\lambda \to \infty} \Theta\left(\lambda(\lambda-\A)^{-1}f\right)=\Theta(f)=0.$$
 Now, since
 $$\Theta\left(\lambda(\lambda-\A)^{-1}f\right)=\lim_{n \to \infty} \lambda \int_{\Gamma_+} \left(M_\lambda H\right)^n G_\lambda f\d\mu_+$$
 we see that, for any $\varepsilon > 0$ we can find $\lambda > 1$ and $N  \geq 1$  such that
\begin{multline}
\|\lambda(\lambda - \mathcal{A})^{-1}\T_\mathrm{max}f - \T_\mathrm{max}f\|_X< \varepsilon \,;\quad \|\lambda(\lambda - \mathcal{A})^{-1}f - f\|_X< \varepsilon\\
\text{ and } \int_{\Gamma_+} \left(M_\lambda H\right)^n G_\lambda f\d\mu_+ < \frac{\varepsilon}{\lambda}\qquad \forall n \geq N.
 \end{multline}
For such $\lambda >1$ and $N \geq 1$, we construct a sequence $(\varphi_n)$ in $\mathscr{W}$ with the following properties
 \begin{multline*}
\B^-\varphi_n = 0; \quad \B^+ \varphi_n = (M_ {\lambda}H)^nG_{\lambda}f\,;\,\quad \|\varphi_n\|_X\leq \|(M_ {\lambda}H)^nG_{\lambda}f\|_{L^1_+}\\
 \quad \text{ and } \qquad \|\T_\mathrm{max} \varphi_n\|_X\leq \|(M_ {\lambda}H)^nG_{\lambda}f\|_{L^1_+}.\end{multline*}
 The existence of  such a sequence is ensured by \cite[Proposition 2.3]{abl2}. Then, for any $n \geq 1$, we set
$$u_n = C_{\lambda}f + \sum_{k=0}^{n-1}\Xi_{\lambda}H(M_{\lambda}H)^kG_{\lambda}f - \varphi_n.$$
Clearly, $u_n\in \D(\T_\mathrm{max})$ for any $n \geq 1$ with
\begin{multline*}
\T_\mathrm{max}u_n = \lambda\left(C_{\lambda}f + \sum_{n=0}^{n-1}\Xi_{\lambda}H(M_{\lambda}H)^nG_{\lambda}f \right) - f - \T_\mathrm{max} \varphi_n;\\
 \B^+ u_n = \sum_{k=0}^{n-1}(M_{\lambda}H)^kG_{\lambda}f\quad \text{ and } \quad \B^- u_n = \sum_{k=0}^{n-1}H(M_{\lambda}H)^kG_{\lambda}f = HB^+ u_n\end{multline*}
i.e. $u_n\in \D(\T_H)$ for all $n \geq 1$. Considering that  $(\lambda - \mathcal{A})^{-1}f = C_{\lambda}f + \sum_{k=0}^{\infty}\Xi_{\lambda}H(M_{\lambda}H)^kG_{\lambda}f$, we can choose $n \geq N$ such that
$\left\|(\lambda - \mathcal{A})^{-1}f - C_{\lambda}f - \sum_{k=0}^{n-1}\Xi_{\lambda}H(M_{\lambda}H)^kG_{\lambda}f\right\|_X < \dfrac{\varepsilon}{\lambda^2}.$
With such choice, since $\lambda(\lambda - \mathcal{A})^{-1}\T_\mathrm{max}f = \lambda^2(\lambda - \mathcal{A})^{-1}f - \lambda f$, we check that
\begin{multline*}\|\lambda u_n - f\|_X \leq \|\lambda (u_n - (\lambda - \mathcal{A})^{-1}f)\|_X + \|\lambda(\lambda - \mathcal{A})^{-1}f - f\|_X
   < 4\varepsilon\\
\text{ and } \qquad \|\lambda \T_\mathrm{max} u_n - \T_\mathrm{max} f\|_X  <  4\varepsilon.\end{multline*}
Since $\lambda u_n \in \D(\T_H)$, this shows that $f\in \mathscr{D}({\,\overline{\T_H} }\,)$.
 \end{proof}

 \subsection{Honesty criteria}
\label{sec5}

Here we want to improve the honesty theory developed in \cite {mmk}. First of all we adapt the definition of honesty, established in the additive perturbation framework in \cite{mmv, almm}.

\begin{definition}\label{defihonest} Let $f\in X_+$ be given. Let $J \subseteq [0,\infty)$ be an interval. Then, the trajectory $(V_H(t)f)_{t\geq 0}$ is said to be honest on $J$ if
$$\left\|V_H(t)f\right\|_X = \left\|V_H(s)f\right\|_X - \widehat{\cc}\left(\int_s^t V_H(r)f \d r\right), \quad \quad \forall \quad s,t\in J, s\leq t.$$
The trajectory is said to be honest if it is honest on $[0,\infty)$. The whole $C_0$-semigroup $(V_H(t))_{t\geq 0}$ will be said honest if all the trajectories are honest.
\label{defhon}
\end{definition}

In the following, we establish thanks to the representation series \eqref{sum} an approach to honesty on subinterval $J \subseteq [0,\infty)$ which is completely new in the context of boundary perturbation. The proof is inspired by the recent similar results obtained in the additive perturbation framework thanks to Dyson-Phillips series (see the concept of so-called 'mild honesty' in \cite[Section 4]{almm}). More precisely, we have the following honesty criteria, analogous to \cite[Theorem 4.8]{almm}:

 \begin{theorem}  \label{CNSH1} Given $f\in X_+$ and $J \subseteq [0,\infty)$, the following statements are equivalent
\begin{enumerate}[1)\:]
\item the trajectory $(V_H(t)f)_{t\geq 0}$ is honest on $J$;\\

\item $\displaystyle \lim_{n\to \infty}\left\|\B^+\int_s^t U_n(r)f \d r\right\|_{L^1_+} = 0$ for any $s,t\in J, s\leq t$;\\

\item $\displaystyle \int_s^t V_H(r)f \d r \in \D(\overline{\T_H}\,)$ for any $s,t\in J, s\leq t$;\\

\item the set $\displaystyle \left(\B^+\int_s^t U_n(r)f \d r\right)_n$ is relatively weakly compact in $L^1_+$ for any $s,t\in J, s\leq t$.
 \end{enumerate}
 \end{theorem}
 \begin{proof} Let $f \in X_+$, $J \subseteq [0,\infty)$ and  $s,t\in J, s\leq t$ be given. Recall that
 $$\cc\left(\int_s^t V_H(r)f \d r\right) = \left\|V_H(s)f\right\|_X - \left\|V_H(t)f\right\|_X.$$
 so that, according to  Definition \ref{defhon},  the trajectory $(V_H(t)f)_{t\geq 0}$ is honest on $J$ if and only if
 $$\Theta\left(\int_s^t V_H(r)f \d r\right) = 0 \quad \forall s,t\in J, s\leq t.$$
 According to \eqref{bart2}, this is equivalent to \textit{2)}, i.e. \textit{1)} $\Leftrightarrow$ \textit{2)}. Since moreover $\int_s^t V_H(r)f \d r \in {\D(\A)_+}$, statements \textit{1)} and  \textit{3)} are equivalent by virtue of Corollary \ref{corhon} and Proposition \ref{prophon}. Clearly  \textit{2)} implies  \textit{4)}. Assume now that the set $\left(\B^+\int_s^t U_n(r)f \d r\right)_n$ is relatively weakly compact in $L^1_+$. Let us show that  $\lim_{n \to \infty}\|\B^+\int_s^t U_n(r)f \d r\|_{L^1_+} = 0$. According to \eqref{bart2}, the limit
$$\lim_{n \to \infty}\left\|\B^+\int_s^t U_n(r)f \d r\right\|_{L^1_+}:= \ell(s,t)$$
exists.
  By Theorem \ref{theotrace}, we also have
 \begin{equation*}\begin{split}
 \left\|\B^+\int_s^t U_n(r)f \d r\right\|_{Y_+} &\leq \left\|\int_s^t U_n(r)f \d r\right\|_{X} + \left\|\T_\mathrm{max} \int_s^t U_n(r)f \d r\right\|_{X}\\
  &= \left\|\int_s^t U_n(r)f \d r\right\|_{X} + \left\|U_n(t)f\right\|_X  - \left\|U_n(s)f\right\|_X,\end{split}\end{equation*}
 and, since  the series $\sum_n \int_s^t U_n(r)f \d r$, $\sum_n U_n(s)f$ and $\sum_n U_n(t)f$ are converging (towards $\int_s^t V_H(r)f \d r$, $V_H(s)f$ and $V_H(t)f$ respectively), one deduces that the right-hand-side is converging to $0$ as $n \to \infty$ and
 \begin{equation}\label{Y+B+}
 \lim_{n\to \infty}\left\|\B^+\int_s^t U_n(r)f \d r\right\|_{Y_+} = 0.
\end{equation}
Now, by assumption \textit{4)}, there exists  a subsequence $\left(\B^+\int_s^t U_{n_k}(r)f \d r\right)_k$ which converges weakly  to, say, $g_{s,t} \in L^1_+$. For any $i \in \mathbb{N}$ we set $\Gamma_{i,+} = \{z \in \Gamma_+\,:\, \tau_-(\z) > \frac{1}{i}\}$ and denote by $\chi_i$ the characteristic function of the set $\Gamma_{i,+}$. Then for any $i \in \mathbb{N}$, the limit
$$  \lim_{k\to \infty}\int_{\Gamma_{i,+}}\B^+\int_s^t U_{n_k}(r)f \d r \d\mu_+ =\lim_{k\to \infty}\int_{\Gamma_{+}} \chi_i(\z)\left(\B^+\int_s^t U_{n_k}(r)f \d r\right)(\z)\d\mu_+(\z) = \int_{\Gamma_{i,+}}g_{s,t}\d\mu_+.$$
Thus, from \eqref{Y+B+},
$$\lim_{k \to \infty} \int_{\Gamma_{i,+}} \B^+\int_s^t U_{n_k}(r)f \d r\,\d\xi_+ = 0 \qquad \forall i \in \mathbb{N}.$$
Since,  for any fixed $i \in N$ and any $\z \in\Gamma_{i,+}$ one has $\d\xi_+(\z) \geq \frac{1}{i}\d\mu_+(\z)$ so that
$$\int_{\Gamma_{i,+}}g_{s,t}\d\mu_+ = 0.$$
 Since $g_{s,t}$ is nonnegative on $\Gamma_+ = \bigcup_{i=1}^{\infty}\Gamma_{i,+}$, we deduce that  $g_{s,t}(\z) = 0$ for $\mu_+$-almost every $\z\in \Gamma_{+}$. In other words, the unique possible weak limit is $g_{s,t}=0$ and therefore $\ell(s,t)=0$, i.e. \textit{2)} holds.\end{proof}
\begin{remark}\label{C0rH1} We deduce directly from the above, with $J=[0,\infty)$ that  the $C_0$-semigroup $(V_H(t))_{t\geq 0}$ is honest if and only if $\lim_{n \to \infty}\|\B^+\int_0^t U_n(s)f \d s\|_{L^1_+} = 0$ for any $f\in X$ and $t > 0$.
 \end{remark}

\begin{remark} Recall that in \cite {mmv}, in the free-streaming case, the defect function $[0,\infty)\ni t \rightarrow \eta_{f}(t)$ has been  defined, for each fixed $f\in (\lambda - \mathcal{A})^{-1}X_+$, by $ \eta_{f}(t):= \|V_H(t)f\| - \|f\| + c_{\lambda}(\int_0^t V_H(s)f \d s)$. We have already observed (see Remark \ref{remmu}) that  $c_{\lambda}$ of \cite {mmv} corresponds to our functional $\widehat{\cc}$. Hence the defect function  can be defined for each fixed $f\in X_+$, as
$$\eta_{f}(t) = - \Theta \left(\int_0^t V_H(s)f \d s\right) = - \lim_{n\to \infty}\left\|\B^+\int_0^t U_n(s)f \d s\right\|_{L^1_+}.$$
 Such a representation of $\eta_{f}$ allows to deduce   immediately  that the mapping $t \rightarrow \eta_{f}(t)$ is nonpositive and nonincreasing. Moreover,  if the trajectory $(V_H(t)f)_{t\geq 0}$ is not honest then there exists $t_0 \geq 0$ such that $\eta_{f}(t) = 0$ for $0\leq t \leq t_0$ and $\eta_{f}(t) < 0$ for all $t > t_0$. Setting $g = V_H(t_0)f \in X_+$. Then for any $t > 0$ one has
\begin{multline*}
\eta_{f}(t + t_0) = -\Theta\left(\int_0^{t + t_0}V_H(s)f \d s\right) \\
= - \Theta \left(\int_0^{t_0}V_H(s)f \d s\right) - \Theta \left(\int_{t_0}^{t + t_0}V_H(s)f \d s\right) = - \Theta\left(\int_0^{t }V_H(s)g \d s\right) = \eta_{g}(t)< 0,\end{multline*}
i.e., with the terminology of \cite{mmv}, the trajectory $(V_H(t)g)_{t\geq 0}$ is immediately dishonest.
\end{remark}

For any subinterval $J \subseteq [0,\infty)$ we denote by
$$X_J:=\{f \in X_+\;;\;(V_H(t)f)_{t \geq 0} \text{ is honest on $J$  } \}$$
and, whenever $J=[0,\infty)$, we simply denote $X_h=X_{[0,\infty)}$ the set of initial positive data giving rise to honest trajectories. Moreover, arguing exactly as in \cite[Proposition 3.13]{almm}, one sees that $X_h$ is invariant under $(V_H(t))_{t \geq 0}$.
Moreover, arguing exactly as in \cite[Proposition 2.4]{mmv}, one has
\begin{proposition} For any subinterval $J \subseteq [0,\infty)$, one has $\widehat{X}_J:=\mathrm{span}(X_J)=X_J - X_J$ is a closed lattice ideal of $X$ whose positive cone is $X_J$. In particular, $\widehat{X_h}=\mathrm{span}(X_h)$ is a closed lattice ideal in $X$ which is invariant under $(V_H(t))_{t \geq 0}$ and $(\widehat{X}_h)_+=X_h.$
\end{proposition}

We recall now that a positive semigroup $\left(T(t)\right)_{t \geq 0}$ in $X$ is said to be irreducible if there is no trivial closed ideal of $X$ (i.e. different from $X$ and $\{0\}$) which is invariant under $T(t)$ for all $t \geq 0$.  We have then the following to be compared to \cite[Theorem 19 \& Remark 20]{mmk}:
\begin{proposition}\label{quasiinterior} Let $g \in X_+$, $g\neq 0$ such that the trajectory $(V_H(t)g)_{t \geq 0}$ is honest.
\begin{enumerate}\item If $(V_H(t))_{t \geq 0}$ is irreducible then the \emph{whole} semigroup $(V_H(t))_{t \geq 0}$ is honest.
\item If $g$ is \emph{quasi-interior} then the \emph{whole} semigroup $(V_H(t))_{t \geq 0}$ is honest.\end{enumerate}
\end{proposition}
\begin{proof} Let $g \neq 0$ such that $(V_H(t)g)_{t \geq 0}$ is honest be given.

\noindent \textit{(1)} One has then $\widehat{X_h} \neq \{0\}$. Since $\widehat{X}_h$ is an ideal invariant under $(V_H(t))_{t \geq 0}$, if  $(V_H(t))_{t \geq 0}$ is irreducible, this shows that necessarily $\widehat{X}_h=X$ and, in particular, $X_+=X_h$.

\noindent \textit{(2)} If $g$ is quasi-interior, since $g \in \widehat{X_h}$ one has $\widehat{X_h}=X$ and the conclusion follows.
\end{proof}

We have the following practical criterion extending \cite[Theorem 2.1 \& Corollary 2.3]{mmk}
\begin{proposition}\label{propH} Assume that there exists some quasi-interior $h \in \lp$ such that
\begin{equation}\label{Hhy} H\,h(\Phi(\z,-\tau_-(\z))) \chi_{\{\tau_-(\z) < \infty\}} \leq h(\z) \qquad \text{ for almost every } \z \in \Gamma_+.\end{equation}
Then, the whole semigroup $(V_H(t))_{t \geq 0}$ is honest.
\end{proposition}
\begin{proof} Let $h \in \lp$ satisfying the above assumption be given. Define then
\begin{equation*}
f(\x) = \begin{cases}
\dfrac{\tau_-(\x)}{\tau_-(\x) + \tau_+(\x)}\exp(-\tau_+(\x))\,h(\Phi(\x,\tau_+(\x)) \qquad &\text{ if } \quad \tau_-(\x) + \tau_+(\x) < \infty,\\
\exp(-\tau_+(\x))\,h(\Phi(\x,\tau_+(\x))  \qquad &\text{ if } \quad  \tau_-(\x)=\infty \quad \text{ and } \tau_+(\x) < \infty,
\end{cases}
\end{equation*}
and $f$ chosen freely on $\O_{+\infty}$ in such a way that $f \in X$ is quasi-interior. One sees easily (see \cite[Proposition 2.3]{abl2} for details) that $\B^+ f=h$. Moreover, since $\t_\pm(\x_t)=\t_\pm(\x) \pm t$ and $\Phi(\x_t,\tau_+(\x_t))=\Phi(\x,\tau_+(\x))$ for any $\x \in \O$, $t > 0$, $\x_t=\Phi(\x,-t)$, one checks easily that, for any $\x \in \O_+$, it holds
\begin{equation*}
U_0(t)f(\x) = \begin{cases}
\dfrac{\tau_-(\x)-t}{\tau_-(\x) + \tau_+(\x)}\exp(-t-\tau_+(\x))\,h(\Phi(\x,\tau_+(\x))\chi_{\{t < \tau_-(\x)\}} \qquad &\text{ if } \x \in \O_+ \cap \O_{-},\\
\exp(-t-\tau_+(\x))\,h(\Phi(\x,\tau_+(\x))  \qquad &\text{ if } \quad  \x \in \O_+ \cap \O_{-\infty}.
\end{cases}
\end{equation*}
Therefore, one sees that for any $t > 0$, $U_0(t)f(\x) \leq f(\x)$ for almost every $\x \in \O_+$. Let $t > 0$ be fixed. According to Lemma \ref{lemmeHBint}, one has
$$\left[ \B^+\int_0^t U_1(s)f \d s \right](\z) \leq  \left[H\B^+\int_0^t U_{0}(s)f\d s\right](\Phi(\z,-\tau_-(\z)) \qquad \forall \z \in \Gamma_+.$$
Since $U_0(s) f \leq f$ on $\O_+$ we get
$$\left[ \B^+\int_0^t U_1(s)f \d s \right](\z) \leq t\,\left[H\,\B^+f\right](\Phi(\z,-\tau_-(\z))=t  H\,h(\Phi(\z,-\tau_-(\z))) \qquad \forall \z \in \Gamma_+.$$
From \eqref{Hhy}, one gets therefore
$$\left[ \B^+ \int_0^t U_1(s)f \d s \right](\z)\chi_{\{\tau_-(\z) < \infty\}} \leq t h(\z) \qquad \text{ for a. e. } \z \in \Gamma_+.$$
Recalling that $\left[ \B^+ \int_0^t U_1(s)f \d s \right](\z)=0$ if $\tau_-(\z) > t$, we get therefore that
$$\left[ \B^+ \int_0^t U_1(s)f \d s \right](\z) \leq t h(\z) \qquad  \text{ for a. e. }  \z \in \Gamma_+.$$
Repeating the argument, one gets that
\begin{equation*}\label{bexp}
\left[\B^+\int_0^t U_n(s)f \d s\right](\z) \leq \dfrac{t^n}{n!} h(\z)\qquad \forall t > 0, n \geq 1,  \text{ for a. e. }\z \in  \Gamma_+ .\end{equation*}
This shows that, for any $t > 0$,
$$\left\|\B^+\int_0^t U_n(s)f \d s\right\|_{L^1_+} \leq \dfrac{t^n}{n!}\|h\|_{L^1_+} \longrightarrow 0 \quad \text{ as } n \to \infty$$
which, according to Theorem \ref{CNSH1}, the trajectory $(V_H(t)f)_{t \geq 0}$ is honest. Since $f$ is quasi-interior, Proposition \ref{quasiinterior} yields the conclusion.\end{proof}
Besides the semigroup approach that we developed in the previous lines, it is also possible to develop a resolvent approach to honesty, as the one  developed in  \cite{mmk} for the free-streaming case and in \cite{all} for conservative boundary conditions.  Such an approach provides  necessary and sufficient conditions for a trajectory to be honest which are different from the one listed above. They can be seen as the analogue of \cite[Theorem 3.5 \&  Theorem 3.11]{almm} which are established in the additive perturbation framework. Since we decided to mainly focus on the semigroup approach, we only state the result for the sake of completeness but omit the details of the proof which can be adapted without major difficulty from   \cite{mmk} and \cite{almm}:

 \begin{theorem} \label{CNSH2} Given $f\in X_+$, the following statements are equivalent
\begin{enumerate}[1)]
\item the trajectory $(V_H(t)f)_{t\geq 0}$ is honest;
\item $\Theta\left((\lambda - \mathcal{A})^{-1}f\right) = 0$ for all/some $\lambda > 0$;
\item $\lim_{n\to \infty}\left\|(M_ {\lambda}H)^nG_{\lambda}f\right\|_{L^1_+} = 0$ for all/some $\lambda > 0$;
\item $(\lambda - \mathcal{A})^{-1}f \in \D({\,\overline{T_H}\,})$ for all/some $\lambda > 0$;
\item the set $\left((M_ {\lambda}H)^nG_{\lambda}f)\right)_n$ is relatively weakly compact in $L^1_+$ for all/some $\lambda > 0$.
\end{enumerate} In particular, the whole $C_0$-semigroup $(V_H(t))_{t\geq 0}$ is honest if and only if $\mathcal{A} = {\overline{T_H}}$.\end{theorem}

\begin{remark}\label{suffhon} It is possible to  provide sufficient conditions for a trajectory to be honest which are reminiscent to those given in \cite[Proposition 2.6]{mmv}. Namely,
\begin{enumerate}
\item given $f \in X_+$, if there exists $\lambda > 0$ such that $(M_{\lambda}H)G_{\lambda}f \leq G_{\lambda}f$, then the trajectory $(V_H(t)f)_{t\geq 0}$ is honest;
\item if $g \in \D(\overline{\T_H})$ is such that $\T_H g \leq \lambda g$ for some $\lambda > 0$, then $g \in X_+$ and  the trajectory $(V_H(t)g)_{t\geq 0}$ is honest.
\end{enumerate}\end{remark}

\section{Some examples}
\label{sec6}

We illustrate here our approach by two examples. These two examples are dealing with the free-streaming equation conservative boundary and, as so, have already been dealt with in our previous contribution \cite{all}. The scope here is to show that our new approach, based upon the semigroup representation \eqref{sum}, allows not only to recover, by different means, the results of \cite{all}, but also to characterize, in both examples, new interesting properties.

\subsection{An instructive one dimensional example revisited}

We revisit here a one-dimensional example introduced in \cite[Example 4.12, p. 76]{voigt}. This example has been revisited recently in both \cite{all,mmk}. Given two real nondecreasing real sequences $(a_k)_{k \geq 0}$ and $(b_k)_{k \geq 0}$ with
$$a_k < b_k < a_{k+1} \qquad \forall k \geq 0, \qquad \lim_{k \to \infty} a_k=\infty$$
set
$$\O= \bigcup_{k=0}^{\infty} (a_k,b_k)=:\bigcup_{k=0}^\infty I_k.$$
We assume then $\mu$ to be the Lebesgue measure on $\mathbb{R}$ and consider the constant field $\mathscr{F}: \mathbb{R}\to \mathbb{R}$ given by $\mathscr{F}(x) = 1$ for all $x \in \mathbb{R}$. In such a case, the flow $\Phi(x,t)$ is given by
$$\Phi(x,t) = x + t \qquad \text{ for any } x, t\in \mathbb{R},$$
 and
$$\Gamma_- = \{a_k, k\in \mathbb{N }\}, \quad \Gamma_+ = \{b_k, k\in \mathbb{N}\}, \quad \tau_-(x) = x - a_k \qquad \forall a_k < x < b_k, \qquad k \in \mathbb{N}.$$
The measures $\d \mu_\pm$ are then the counting measures over $\Gamma_\pm$. We define then $H \in \Be(\lp,\lm)$ by
\begin{equation}
\label{boundh}
H\psi(a_k) = \begin{cases}  0 \qquad &\text{ if } k = 0,\\
b_{k-1} \qquad &\text{ if } k > 0
\end{cases}
\end{equation}
for any $\psi \in \lp$. It is clear that  $H$ is a positive boundary operator with unit norm. We then explicit the strongly family of operators $\left\{(U_k(t))_{t \geq 0}\,;\,k \in \mathbb{N}\right\}$ as defined in Definition \ref{UK}.  To this aim for any $k\in \mathbb{N}$, set $\Delta_k = b_k - a_k$. For $f\in \mathcal{D}_0$ and $t > 0$ one easily sees that
\begin{equation}
U_0(t)f(x) =\begin{cases}
 f(x - t) \qquad &\text{ if } 0 < t < x - a_k,\\
0 \qquad &\text{
otherwise},\end{cases}
\label{uont}
\end{equation}
which yields
\begin{equation}
\B^+U_0(t)f(b_k) = \begin{cases}
 f(b_k - t) \qquad &\text{ if } 0 < t < \Delta_k,\\
0 \qquad &\text{
otherwise}.\end{cases}
\label{buot}
\end{equation}
By induction one can easily show that for $n\geq 1$, $k \geq 0$, $a_k < x < b_k$ one has
\begin{equation}
U_n(t)f(x) = \begin{cases}
 f(b_{k - n} - a_{k} + x + &\sum_{j=k-n+1}^{k-1}\Delta_{j} - t)\qquad \text{ if } k\geq n \\
 \qquad &\text{  and}\quad x - a_{k} + \sum_{j=k-n+1}^{k-1}\Delta_{j} < t < x - a_{k} + \sum_{j=k-n}^{k-1}\Delta_{j},\\
0 \qquad &\text{
otherwise}\end{cases}
\label{bunt}
\end{equation}
so that
\begin{equation}
\B^+U_n(t)f(b_k) = \begin{cases}
f(b_{k - n} + \sum_{j=k-n+1}^{k}\Delta_{j} - t)\quad &\text{ if }\quad k\geq n \quad \text{and}\quad\sum_{j=k-n+1}^{k}\Delta_{j} < t < \sum_{j=k-n}^{k}\Delta_{j},\\
0 \qquad &\text{
otherwise}.\end{cases}
\label{bbnt}
\end{equation}
Because of this we have  for all $f\in X$
\begin{equation}
\B^+\left(\int_0^tU_n(s)f \d s\right)(b_k) = \begin{cases}
\displaystyle\int_{a_{k-n}\vee(b_{k-n}+\sum_{j=k-n+1}^{k}\Delta_{j}-t)}^{b_{k-n}}f(s)\d s \qquad &\text{ if }  \quad k\geq n \quad \text{ and}\quad t > \sum_{j=k-n+1}^{k}\Delta_{j},\\
0 \qquad &\text{
otherwise}.\end{cases}
\label{btint}
\end{equation}
Now we are able to prove the following where $(V_H(t))_{t \geq 0}$ is the $C_0$-semigroup constructed through Theorem \ref{limiting} and given by  $V_H(t) = \sum_{n=0}^{\infty}U_n(t)$ ($t\geq 0$):
 \begin{proposition}
 The $C_0$-semigroup $(V_H(t))_{t\geq 0}$ is honest if and only if
 \begin{equation}
 \Delta := \sum_{k=0}^{\infty}(b_k - a_k) = \infty
 \label{honsem}
 \end{equation}
 If $\Delta < \infty$, define
 $$J_k := \bigg[\sum_{j=k+1}^{\infty}\Delta_j, \sum_{j=k}^{\infty}\Delta_j\bigg]\subset [0,\infty) \qquad \text{ for any } \quad k \in \mathbb{N}.$$ Then, given $f\in X_+$, the trajectory $(V_H(t)f)_{t\geq 0}$ is honest on $J_k$ if and only if $\displaystyle\int_{a_{k}}^{b_{k}}f(s)\d s = 0$ which is equivalent to $f(s) = 0$ for almost every $s\in I_{k}$.
\label{honhn}
\end{proposition}
\begin{remark} The first part of the Proposition is a well-known fact, first proven in \cite{voigt} and revisited recently in \cite{all, mmk}. The second part, on the contrary, is new to our knowledge and provides a criterion for 'local' honesty.
\end{remark}

\begin{proof}
Thanks to formula (\ref{btint}) we can state that for all $f\in X_+$, $n \geq 1$ and $t > 0$
 \begin{equation}\begin{split}
\left\|\B^+\int_0^t U_n(s)f \d s\right\|_{L^1_+} &= \sum_{k=0}^{\infty}\left[\int_{a_{k}\vee(b_{k}+\sum_{j=k+1}^{k+n}\Delta_{j}-t)}^{b_{k}}f(s)\d s\right]\chi_{\{t>\sum_{j=k+1}^{k+n}\Delta_{j}\}}\\
&\leq \sum_{k=0}^{\infty}\int_{a_{k}}^{b_{k}}f(s)\d s = \|f\|_X.
\label{bisem}
\end{split}\end{equation}
Assume first that $ \Delta = + \infty$. In such a case for any $f\in X_+$, $k \in \mathbb{N}$  we have:
$$\lim_{n\to \infty}\left[\int_{a_{k}\vee(b_{k}+\sum_{j=k+1}^{k+n}\Delta_{j}-t)}^{b_{k}}f(s)\d s\right]\chi_{\{t>\sum_{j=k+1}^{k+n}\Delta_{j}\}} = 0 $$
which, thanks to the dominated convergence theorem yields
$$\lim_{n\to \infty}\left\|\B^+\int_0^t U_n(s)f \d s\right\|_{L^1_+} =0.$$
By virtue of Corollary \ref{C0rH1} the trajectory $(V_H(t)f)_{t\geq 0}$ is then honest. Since $f\in X_+$ is arbitrary, the semigroup itself is honest.

Now consider the case $ \Delta < + \infty$. In such a case, given $k_0 \in \mathbb{N},$ let $t\in J_{k_0} \subseteq [0,\Delta]\subseteq [0,\infty)$. It is easy to see that
$$
\lim_{n\to \infty}\left[\int_{a_{k}\vee(b_{k}+\sum_{j=k+1}^{k+n}\Delta_{j}-t)}^{b_{k}}f(s)\d s\right]\chi_{\{t>\sum_{j=k+1}^{k+n}\Delta_{j}\}} = \begin{cases}
0 & \text{  if }  k = 0,\ldots ,k_0 - 1 \geq 0,\\
\displaystyle \int_{a_{k}}^{b_{k}}f(s) \d s & \text{ if}  k \geq k_0 + 1,
\end{cases}
$$
while

$$\lim_{n\to \infty}\left[\int_{a_{k_0}\vee(b_{k_0}+
\sum_{j=k_0+1}^{k_0+n}\Delta_{j}-t)}^{b_{k_0}}f(s)\d s\right]\chi_{\{t>\sum_{j=k_0+1}^{k_0+n}\Delta_{j}\}} = \int_{(b_{k_0}+\sum_{j=k_0+1}^{\infty}\Delta_{j}-t)}^{b_{k_0}}f(s)\d s.$$

Then, from the dominated convergence theorem and \eqref{bisem}) we get

$$\lim_{n\to\infty}\left\|\B^+\int_0^t U_n(s)f \d s\right\|_{L^1_+} = \int_{(b_{k_0}+\sum_{j=k_0+1}^{\infty}\Delta_{j}-t)}^{b_{k_0}}f(s)\d s + \sum_{k=k_0+1}^{\infty}\int_{a_{k}}^{b_{k}}f(s)\d s.$$
Therefore, for any $t_1$, $t_2 \in J_{k_0}$ with $t_1 < t_2$ one has

$$\lim_{n\to\infty}\left\|\B^+\int_{t_1}^{t_2} U_n(s)f \d s\right\|_{L^1_+} = \int_{(b_{k_0}+\sum_{j=k_0+1}^{\infty}\Delta_{j}-t_2)}^{(b_{k_0}+\sum_{j=k_0+1}^{\infty}\Delta_{j}-t_1)}f(s)\d s \leq \int_{a_{k_0}}^{b_{k_0}} f(s)\d s$$
where the last inequality is an identity if $t_1 = \sum_{j=k_0+1}^\infty \Delta_j$ and $t_2 =\sum_{j=k_0}^\infty \Delta_j$.
Using Corollary \ref{C0rH1} again, this shows that the trajectory $(V_H(t)f)_{t\geq 0}$ is honest on $J_{k_0}$
 if and only if $\int_{a_{k_0}}^{b_{k_0}}f(s)\d s = 0$ and, being $f$ nonnegative, this is equivalent to $f(s) = 0$ for almost every $s\in I_{k_0}$.\end{proof}

\begin{remark}
 As an immediate consequence of the obtained result we can state the following: if $\Delta < \infty$, then no trajectory $(V_H(t)f)_{t\geq 0}$ ($f\in X_+$) is honest on an interval $J \supseteq [0,\Delta]$.
 \end{remark}
\begin{remark} Finally observe that, in case $ \Delta < + \infty$, for all $f\in X$ and $t \geq \Delta$ one has $V_H(t)f\equiv 0$. Indeed for $t \geq \Delta$ one has $U_0(t)f\equiv 0$. Furthermore for any $n \geq 1$, $k \geq 0$, one has $t > \sum_{j=k}^{k+n}\Delta_{j}$. This implies $U_n(t)f(x) = 0$ for all $n \geq 1$, $f\in \mathcal{D}_0$ and $x\in\O$ which gives the result.
\end{remark}
\subsection{Kinetic equation with specular reflections}

In this second example, we consider the physically relevant case of free-streaming semigroup associated to specular reflections. Such a model, as well-known \cite{al3}, is strongly related to the so-called billiard flow which is a well-known dynamical system studies in ergodic theory \cite{corn}. We do not provide here any new honesty criterion but show how the result we obtained before yields possibly new property of the billiard flow. More precisely, we consider now a transport equation in $\mathbb{R}^N$ with $N = 2d$, $d \in \mathbb{N}$ and consider then
$$\O = D \times \mathbb{R}^d$$
 where  $D$ is a smooth open bounded and convex subset of $\mathbb{R}^d$. Any $\x \in \O$ can be written $\x=(x,v)$, with $x \in D$, $v \in \mathbb{R}^d$ and consider the measure $\d \mu(\x) = \d x\otimes \d\varrho(v)$, where $\d \varrho$ is a positive Radon measure on $\mathbb{R}^d$ with support $V$. We assume for simplicity that $V$ and $\d\varrho$ to be orthogonally invariant. Assume  the field $\mathscr{F}\::\:\x \in \mathbb{R}^N \to \mathbb{R}^N$  to be given by $\mathscr{F}(\x) = (v,0)$ for all $\x=(x,v)\in R^d\times R^d$. Classically, the associated flow is given by $\Phi(\x,t) = (x + vt, v)$ for all $\x=(x,v)$ and $t\in \mathbb{R}$. In this case,
 $$\Gamma_{\pm}=\left\{\y=(x,v) \in \partial  D \times V\,;\, \pm v \cdot n(x) > 0\right\}, \qquad \d\mu_{\pm}(\y)=|v \cdot n(x)|\d\varrho(v)\d\sigma(x)$$
 where  $n(x)$ denotes the outward normal unit vector at $x \in \partial D$ and $\d\sigma(\cdot)$ is the Lebesgue surface measure on $\partial D$. We consider here the boundary operator $H$ to be  associated by the specular reflection, i.e.
$$H\psi(\y) = \psi(x, v - 2(v\cdot n(x))n(x)), \qquad \psi \in L^1_+, \:\;\y=(x,v)\in\Gamma_-.$$
It is known that $H$ is a positive and conservative operator, i.e. $\|H\psi\|_{L^1_-} = \|\psi\|_{L^1_+}$ for any nonnegative $\psi \in L^1_+$. In particular, $H$ has unit norm.  As in the previous example, let us characterize the families $(U_k(t))_{t\geq 0}$, $k \in \mathbb{N}$. Observe  that for $\x = (x,v) \in \overline{D}\times V$ we can define the sequence of \emph{rebound times}:
$$t_1(\x) = \tau_-(\x),\quad \quad t_2(\x) = t_1(\x) + \tau_-(\x_1) \quad \ldots \quad t_k(\x) = t_{k - 1}(\x) + \tau_-(\x_{k-1}) = \sum_{j=1}^{k-1}\tau_-(\x_j)$$
where , setting $\x_0 = \x$, $x_0 = x$, $v_0 = v$, one has, for any $j = 1,\ldots,k$:
\begin{multline*}
\x_j = (x_j,v_j )\in \Gamma_+, \qquad  \text{ with } \quad x_j = x_{j-1} - \tau_-(\x_{j-1})v_{j-1}, \in \partial D,\\
  v_j = v_{j-1} - 2(v_{j-1}\cdot n(x_{j-1}))n(x_{j-1}).\end{multline*}

With this notations, setting also $t_0 = 0$, we have for any $f\in \mathcal{D}_0$,  $k \geq 0$, $\x = (x,v)\in\O$
$$
U_k(t)f(\x) =
f(x_k - (t - t_k(\x))v_k,v_k)\,\chi_{\{t_k(\x) \leq t < t_{k+1}(\x)\}}.$$
Recall that $V_H(t)=\sum_k U_k(t)$ and, since for a given $t > 0$ and a given $\x \in \O$, there exists a unique $k \in \mathbb{N}$ such that $t \in [t_k(\x),t_{k+1}(\x)]$, one has $V_H(t)f(\x)=U_k(t)f(\x)$ (of course, such $k$ depends on $\x$). This implies that, for any $t\geq 0$,
$$V_H(t)f = f\circ \vartheta_t$$
 where $\{\vartheta_t; t \in \mathbb{R}\}$ is the one-parameter group of transformations on $\overline{D}\times V$ corresponding to the so-called \emph{billiard flow} \cite {corn} (see also \cite {abl4}).  The following is taken from \cite{voigt, all} and is proven through the resolvent approach:
 \begin{proposition} Assume that $\d\sigma(\partial D) < \infty$. If there exists some nonnegative $\psi \in L^1(V,\d\varrho(v))$ with $\psi(v)=\psi(|v|)$ for any $v \in V$ and $$\int_{V} \left(1+|v|\right)\psi(|v|)\d\varrho(v) < \infty$$
 then the semigroup $(V_H(t))_{t \geq 0}$ is honest.
 \end{proposition}
 \begin{proof} The proof follows from Proposition \ref{propH} since, as in \cite[Corollary 2.3]{mmk}, the mapping $h(x,v)=\psi(|v|)$ provides a quasi-interior element of $\lp$ satisfying \eqref{Hhy}.
 \end{proof}
 \begin{remark} From Remark \ref{C0rH1} we deduce that for any $t > 0$ and $f\in X_+$
 $$\lim_{k \to \infty}\left\|\B^+\int_0^t U_k(s)f \d s\right\|_{L^1_+} = 0,$$
 i.e.
 $$\lim_{k\to \infty}\int_{\Gamma_+}\left(\int_{t_k(\x)}^{t\wedge t_{k+1}(\x)}f(x_k - v_k(s - t_k(\x))v_k,v_k)\d s\right)\chi_{\{t_k(\x) < t\}}\d\mu_+(\x) = 0.$$
We wonder if such a property of the billiard flow is known in the literature.\end{remark}



\thebibliography{AAAA}
\renewcommand{\baselinestretch}{1}
\bibitem {abl1} {\sc L. Arlotti, J. Banasiak, B. Lods}, A new approach to transport
equations associated to a regular field: trace results and well-posedness,
{\em Mediterr. J. Math.}, {\bf 6} (2009), 367-402.
\bibitem {abl2} {\sc L. Arlotti, J. Banasiak, B. Lods}, On general transport
equations with abstract boundary conditions. The case of divergence free force field,
{\em Mediterr. J. Math.}, {\bf 8} (2011),1-35.
\bibitem {al3} {\sc L. Arlotti}, Explicit transport semigroup associated to abstract boundary conditions  \emph{Discrete Contin. Dyn. Syst.,} Dynamical systems, differential equations and applications. Suppl. Vol. I,  (2011) 102--111.
\bibitem {abl4} {\sc L. Arlotti}, Boundary conditions for streaming operator in a bounded convex body, {\em Transp. Theory Stat. Phys.}, {\bf 15} (1986), 959-972.
\bibitem {all} {\sc L. Arlotti,  B. Lods}, Substochastic semigroups for transport equations with conservative boundary conditions {\em J. Evolution Equations}, {\bf 5} (2005), 485-508.
\bibitem {almm} {\sc L. Arlotti, B. Lods and M. Mokhtar-Kharroubi}, On perturbed substochastic semigroups in abstract state spaces {\em Z. Anal. Anwend.}, \textbf{30} (2011), 457-495.
\bibitem {almm1} {\sc L. Arlotti, B. Lods and M. Mokhtar-Kharroubi}, Non-autonomous Honesty
theory in abstract state spaces with applications to linear kinetic equations, {\em Commun. Pure Appl. Anal.}, {\bf 13} (2014), 729--771.
\bibitem {almm2} {\sc L. Arlotti, B. Lods and M. Mokhtar-Kharroubi}, On perturbed substochastic once integrated semigroups in abstract state spaces, in preparation.

\bibitem {ba} {\sc J. Banasiak, L. Arlotti}, Perturbations of Positive Semigroups with Applications,
{\em Springer Monographs in Mathematics}, 2006.
\bibitem {bard} {\sc C. Bardos}, Probl\`{e}mes aux limites pour les \'{e}quations
aux d\'{e}riv\'{e}es partielles du premier ordre \`{a} coefficients r\'{e}els; th\'{e}or\`{e}mes
d'approximation, application \`{a} l'\'{e}quation de transport,
{\em Ann. Sci. \'{E}cole Norm. Sup.}, {\bf 3} (1970), 185--233.
\bibitem {beals} {\sc R. Beals, V. Protopopescu}, Abstract Time-dependent Transport Equations
{\em J. Math. Anal. Appl.}, {\bf 121} (1987), 370--405.
\bibitem{boulanouar}
{\sc M. Boulanouar}, New results in abstract time-dependent transport equations  {\em Transport Theory Statist. Phys.},  \textbf{40} (2011), 85--125.

\bibitem{cer} {\sc C. Cercignani}, The Boltzmann Equation and its Applications, {\em Springer Verlag}, 1988.

\bibitem{corn} {\sc I. P. Cornfeld, S. V. Fomin and Ya. G. Sinai}, Ergodic Theory, {\em Springer Verlag}, 1982.
\bibitem{dautray} {\sc R. Dautray, J. L. Lions}, Mathematical analysis
and numerical methods for science and technology. Vol. 6: Evolution
problems II, {\em  Berlin, Springer}, 2000.

\bibitem{Engel}
{\sc K.~J. Engel and R. Nagel,}  \textit{One-parameter semigroups for
linear evolution equations}, Springer, New-York, 2000.


\bibitem{lods} \textsc{B. Lods}, Semigroup generation properties of
streaming operators with noncontractive boundary conditions, {\it
Math. Comput. Modelling} {\bf 42} (2005) 1441--1462.


\bibitem {mmk} {\sc M. Mokhtar-Kharroubi}, On collisionless transport semigroups with boundary operators of norm one, {\em J. Evolution Equations}, {\bf 8} (2008), 327-362.
\bibitem {mmv} {\sc M. Mokhtar-Kharroubi, J.Voigt}, On honesty of perturbed substochastic $C_0$-semigroups in $L^1$-spaces, {\em J. Operator Theory}, (2009).
\bibitem {voigt} {\sc J. Voigt}, Functional analytic treatment of the initial boundary value for collisionless
gases, {\em M\"{u}nchen, Habilitationsschrift}, 1981.
\end{document}